\documentclass[leqno,a4paper,twoside]{article}
\usepackage{amsmath,amsthm,amsxtra,amssymb}
\usepackage{graphics}
\usepackage{graphicx}
\usepackage{mathrsfs} %for\mathscr

\theoremstyle{plain}
\newtheorem{theorem}{Theorem}[subsection]
\newtheorem{prop}[theorem]{Proposition}
\newtheorem{defs}[theorem]{Def\mbox{}inition}
\newtheorem{lemma}[theorem]{Lemma}
\newtheorem{coro}[theorem]{Corollary}
\newtheorem{rem}[theorem]{Remark}

\newtheorem{atheorem}{Theorem}[section]
\newtheorem{aprop}[atheorem]{Proposition}
\newtheorem{adefs}[atheorem]{Def\mbox{}inition}
\newtheorem{alemma}[atheorem]{Lemma}

\theoremstyle{remark}
\newtheorem{exa}[theorem]{Example}

\numberwithin{equation}{section}

\newcommand{\Sswj}{{ {\mathcal S}{\mathcal W}^*_A}}
\newcommand{\Ssawj}{{ {\mathcal S}{\mathcal W}^*_{\{A_\alpha\}}}}
\newcommand{\Swj}{{\mathcal W}^*_A}
\newcommand{\Volo}{{\underset{ {\mathcal A} }{\otimes} {\mathcal V}_X}}

\newcommand{\mc}[1]{\mathcal {#1}}
\newcommand{\ms}[1]{\mathscr {#1}}

\newcommand{\mrm}[1]{\mathrm{#1}}
\renewcommand{\phi}{\varphi}
\renewcommand{\theta}{\vartheta}
\renewcommand{\rho}{\varrho}
\newcommand{\ep}{\epsilon}
\newcommand{\lra}{\longrightarrow}
\renewcommand{\bigr}[1]{{\big(#1\big)}}

\renewcommand{\biggr}[1]{{\bigg(#1\bigg)}}
\newcommand{\bigc}[1]{{\big\{#1\big\}}}

\newcommand{\biggc}[1]{{\bigg\{#1\bigg\}}}
\newcommand{\ou}[3][]{\overset{{#1}}{\underset{{#2}}{{#3}}}}

\newcommand{\Mod}{\mathrm{Mod}}

\newcommand{\com}{\mathbb{C}}
\newcommand{\rea}{\mathbb{R}}

\newcommand{\integer}{\mathbb{Z}}
\newcommand{\integergz}{\mathbb{Z}_{>0}}
\newcommand{\nat}{\mathbb{N}}

\newcommand{\Op}{\mathrm{Op}}

\newcommand{\opxsac}{\mathrm{Op}^c(\xsa)}
\newcommand{\cov}[1]{\mathrm{Cov}_{sa}(#1)}
\newcommand{\xsa}{{X_{sa}}}

\newcommand{\cinfty}{\mathscr{C}^\infty}

\newcommand{\D}{\mathcal{D}}

\newcommand{\proofend}{\hfill $\Box$ \vspace{\baselineskip}\newline}
\usepackage{fancyhdr}
\author{N. Honda, G. Morando}
\title{Stratif\mbox{}ied Whitney jets and tempered ultradistributions on the subanalytic site}
\date{\empty}

\long\def\symbolfootnote[#1]#2{\begingroup%
\def\thefootnote{\fnsymbol{footnote}}\footnote[#1]{#2}\endgroup}

\newpage

\begin{document}

\maketitle

\begin{abstract}
  In this paper we introduce the sheaf of stratif\mbox{}ied Whitney
  jets of Gevrey order on the subanalytic site relative to a real
  analytic manifold $X$. Then we def\mbox{}ine stratif\mbox{}ied
  ultradistributions of Beurling and Roumieu type on $X$. In the end,
  by means of stratif\mbox{}ied ultradistributions, we def\mbox{}ine
  tempered-stratif\mbox{}ied ultradistributions and we prove two
  results. First, if $X$ is a real surface, the tempered-stratif\mbox{}ied
  ultradistributions def\mbox{}ine a sheaf on the subanalytic site relative
  to $X$. Second, the tempered-stratif\mbox{}ied ultradistributions on the
  complementary of a $1$-regular closed subset of $X$ coincide with
  the sections of the presheaf of tempered ultradistributions.
\end{abstract}

\symbolfootnote[0]{\phantom{a}\hspace{-7mm}\textit{2000 MSC.} Primary 46M20
  ; Secondary 46F05 32B20 32C38.} 

\symbolfootnote[0]{\phantom{a}\hspace{-7mm}\textit{Keywords and
    phrases:} sheaves on subanalytic sites, tempered
  ultradistributions, Whitney jets.}

\vspace{-5mm}

\tableofcontents

\addcontentsline{toc}{section}{\textbf{Introduction}}
\section*{Introduction}
%\sectionmark{Introduction}
%\input{introduction.tex}
\markboth{Introduction}{}

One of the aim of the present article is to def\mbox{}ine tempered
ultradistributions of Beurling and Roumieu class of order $s>1$ and
Whitney jets with growth conditions as sheaves on the subanalytic site
relative to a real analytic manifold $X$. As growth conditions are not of local
nature, functional spaces def\mbox{}ined on open subsets of $X$, as tempered distributions, Whitney $\ms
C^\infty$-functions or holomorphic functions with polynomial growth
at the boundary do not glue on arbitrary coverings. In particular,
such spaces do not def\mbox{}ine sheaves on the usual topology of an analytic
manifold.  We recall the approach set by S. \L ojasiewicz  
(\cite{lojasiewicz}) later reinterpreted and generalized in the works of
M. Kashiwara and P. Schapira (see \cite{kashiwara_riemann-hilbert}, \cite{ks_moderate_formal_cohomology}
and \cite{ks_indsheaves}). They def\mbox{}ined tempered distributions and
Whitney $\ms C^\infty$-functions as sheaves on the subanalytic site, $\xsa$, 
relative to a real analytic manifold $X$. The open sets of $\xsa$ are
the relatively compact subanalytic open subsets of $X$ and 
the coverings are the locally f\mbox{}inite ones. The use of these objects in the
study of linear ordinary dif\mbox{}ferential equations gave
interesting results (see
\cite{morando_tempered_solutions_formal_invariants}). Let us mention
that function spaces with growth conditions, such as holomorphic
functions on the complex plane with moderate or Gevrey growth or
asymptotic expansion at the
origin, are treated as sheaves on the real blow up at the origin by
B. Malgrange in \cite{malgrange_birkhauser} and many other authors
elsewhere in litterature. Such function spaces are used in a
systematic way in the study of linear ordinary dif\mbox{}ferential 
equations. Some of these sheaves on the real blow up at the origin can be
obtained by specializing their subanalytic generalization (see
\cite{prelli_microlocalization_subanalytic_sheaves}).

Among the motivations of this paper there is the fact that the naive
def\mbox{}inition of tempered ultradistributions, mimicking that of tempered
distributions (see \cite{kashiwara_riemann-hilbert}), does not give a sheaf on the
subanalytic site, as explained in Section \ref{section_example}. 
Let us recall that tempered ultradistributions on an open set $U$ in
$X$ are def\mbox{}ined as global sections of ultradistributions modulo
ultradistributions with support on $X\setminus U$. This latter space
is the dual of Whitney jets with Gevrey like growth conditions on
$X\setminus U$. In this paper, we relax the condition on Whitney jets
with Gevrey like growth conditions by introducing the stratif\mbox{}ied Whitney jets on a
real analytic manifold $X$. We prove decomposition and gluing properties for stratied
Whitney jets on locally f\mbox{}initely many subanalytic subsets of $X$ (Lemma 
\ref{lemma:stratified-Whitney-exact}). Then we study the dual of
stratif\mbox{}ied Whiteny jets on a closed set $Z\subset X$, the space of
stratif\mbox{}ied ultradistributions on $Z$. This latter space is a subspace of
ultradistribution with support in $Z$. We study the decomposability of
stratif\mbox{}ied ultradistributions on arbitrary f\mbox{}initely many subanalytic
closed sets (Corollary \ref{coro:decomposability_sdb} and Corollary
\ref{coro:ultra-exact}). Then we def\mbox{}ine tempered-stratif\mbox{}ied
ultradistributions on $U$ as global ultradistributions modulo stratif\mbox{}ied
ultradistributions on $X\setminus U$. We prove that, when $X$ has
dimension $2$, tempered-stratif\mbox{}ied ultradistributions def\mbox{}ine a sheaf on
$\xsa$. Further, we prove that, if $X\setminus U$ satisf\mbox{}ies a
regularity condition, tempered-stratif\mbox{}ied ultradistributions on $U$
coincides with classical tempered ultradistributions on $U$ (Theorem
\ref{thm_qt_versus_t}). We conclude by proving that tempered-stratif\mbox{}ied
ultradistributions and other spaces of ultradistributions similarly
def\mbox{}ined do not give rise to sheaves on $\xsa$, if $X$ has dimension $>2$. 

Similar results on the decomposability of ultradistributions were obtained
by J.-M. Kantor (\cite{kantor}) and by A. Lambert (\cite{lambert}). Their
approach is quite dif\mbox{}ferent from 
our. Indeed, given $s>1$, they f\mbox{}ind a family $\mc T_s$ of subanalytic closed 
sets depending on $s$ such that ultradistributions of class $s$ decompose on sets
in $\mc T_s$. The family $\mc T_s$ is not closed under intersections
hence it is not possible to def\mbox{}ine a Grothendieck topology and a notion
of sheaf starting from it.

In the end, let us recall that ultradistributions and growth
conditions of Gevrey type turned out to 
be very useful in the functorial study of linear dif\mbox{}fential equations, being
strictly linked to the irregularity of equations. Let us cite, for example,
\cite{honda_reconstruction} and \cite{yamazaki} for some applications of
ultradistributions in the study of systems of linear dif\mbox{}ferential
equations. In the present article we do not use tempered-stratif\mbox{}ied ultradistributions to
study systems of linear dif\mbox{}ferential equations, postponing this
problem to future investigations. Throughout the paper, we just limit
to point out if the sheaves we def\mbox{}ine give rise to sheaves of
modules over the ring of linear dif\mbox{}ferential operators with analytic
coef\mbox{}f\mbox{}icients.

The paper is organized as follows. We start {\bf Section
\ref{section:recall}} by recalling the basic properties of
Whiteny jets with growth conditions. Then, mimicking
\cite{kashiwara_riemann-hilbert}, we def\mbox{}ine the presheaf of
tempered ultradistributions and we recall a condition, due to H. Komatsu,
for a continuous function to extend to the whole space as an
ultradistribution. In the end of the section, we prove that tempered
ultradistributions do not glue on f\mbox{}initely many subanalytic open
subsets of $\rea^2$.

In {\bf Section \ref{section:swj}} we start by recalling some
def\mbox{}initions and basic results on subanalytic sets and the subanalytic
site relative to a real analytic manifold $X$. Then, we def\mbox{}ine the space
of stratif\mbox{}ied Whitney jets with Gevrey growth conditions and we prove
that they give rise to a sheaf on the subanalytic site relative to
$X$. Then, we introduce the space of stratif\mbox{}ied ultradistributions on
$X$ and we prove that this space is dual to stratif\mbox{}ied Whitney
jets. In the end of the section, from the gluing property of
stratif\mbox{}ied Whitney jets, we obtain a decomposition property for
stratif\mbox{}ied ultradistributions.  

In {\bf Section \ref{section:2dim}} given a real analytic manifold
$X$, we def\mbox{}ine tempered-stratif\mbox{}ied ultradistributions on a subanalytic open
set $U\subset X$ which is a subspace of tempered ultradistributions on
$U$. Then, we prove two results. The f\mbox{}irst states that, if $\dim X=2$,
tempered-stratif\mbox{}ied ultradistributions def\mbox{}ine a sheaf on the subanalytic
site relative to $X$. The second states that if $X\setminus U$
satisf\mbox{}ies a regularity condition, then tempered-stratif\mbox{}ied
ultradistributions on $U$ coincide with tempered ultradistributions on
$U$.

In {\bf Appendix \ref{section:sgi}} we prove a result of density for
stratif\mbox{}ied Whitney jets in the space of Whitney jets. Such results is
needed in Section \ref{section:swj}, we prove it in the Appendix
as the proof is rather long and technical.
% and it relies on the study of super
%growth indicators.

\section{Notations and review on Whitney jets and ultradistributions}\label{section:recall}

In this paper, we assume that a real analytic manifold is
countable at inf\mbox{}inity.

\subsection{Whitney jets with Gevrey conditions}

Let $X$ be a real analytic manifold. We denote by $\Mod(\com_X)$ the
category of sheaves on $X$ with values in $\com$-vector spaces,
and by $\cinfty$ the sheaf of infinitely differentiable functions on $X$.
We denote by $\pi_k: J^k \to X$ ($k\in\integer_{\geq0}$) 
the vector bundle associated with k-th jets over $X$. 
For any non-negative integers $k_1 \ge k_2$,
the morphism of vector bundles $j^{k_2,k_1}: J^{k_1} \to J^{k_2}$ 
is def\mbox{}ined by the canonical projection from $k_1$-th jets to
$k_2$-th jets.

%\

Let $A$ be a locally closed subset in $X$, and $J^k(A)$ designates the
set of continuous sections of the vector bundle $J^k$ over $A$.  We
denote by $j^k_X: {\cinfty}(X) \to J^k(X)$ the canonical jets
extension morphism, and for any locally closed spaces $A \subset B$,
we designate by $j^k_{A,B}: J^k(B) \to J^k(A)$ the natural restriction
map from sections over $B$ to those over $A$. Composing $j^k_X$ and
$j^k_{A,X}$ we have the canonical morphism
$$
j^k_A = j^k_{A,X}\circ j^k_X: \cinfty(X) \to J^k(A).
$$

%\

The morphism of vector bundles $j^{k_2,k_1}$ induces the
map
$$
j^{k_2,k_1}_A: J^{k_1}(A) \to J^{k_2}(A),
$$
and using these maps we def\mbox{}ine the jets space over $A$ by
$$
J(A) := \ou{k\in\integer_{\geq0}}{\varprojlim} J^k(A).
$$
The morphism $j^k_X$ (resp. $j^k_{A,B}$) induces
\begin{eqnarray*}
j_X&:& \cinfty(X) \to J(X)\\
(\text{resp.}\quad j_{A,B}  &:& J(B) \to J(A)) \ .
\end{eqnarray*}
We set $j_A := j_{A,X} \circ j_X: \cinfty(X) \to J(A)$.

If $X=\mathbb R^n$ with a system of coordinates $(x_1, x_2, \dots,
x_n)$, then the jets space $J(A)$ is isomorphic to the set
$$
\left\{\{f_\alpha\}_{\alpha \in ({\mathbb Z}_{\ge 0})^n};\, f_\alpha \in \ms
C^0(A)\right\} $$
% where $C(A)$ denotes the set of continuous functions in $A$, and
where $\ms C^0(A)$ designates the set of continuous functions on $A$.
The map $j_A$ is identif\mbox{}ied with
$$
j_A(f) = \biggc{
\left. \frac{\partial^\alpha f}{\partial x^\alpha}\right|_A }_{\alpha \in (\integer_{\ge0})^n} \qquad
\text{for }
f \in\cinfty(X).
$$

\

Let $A$ be a locally closed subset in $X$. We def\mbox{}ine ${\mathcal
  J}_A\in\Mod(\com_X)$ by
$$
{\mathcal J}_A(U) := J(A \cap U)
$$
for $U$ an open subset of $X$.  If $A \cap U = \varnothing$, then we
consider $J(A \cap U)$ as the zero object.  The morphism $j_{A \cap U, U}
\circ j_U: \cinfty(U) \to J(A\cap U)$ induces the sheaf homomorphism
$j_A: \cinfty \to {\mathcal J}_A$.

We have that ${\mathcal J}_A$ is a sheaf of rings and modules over
$\D_X$: the sheaf of rings of linear dif\mbox{}ferential operators
with analytic coef\mbox{}ficients on $X$.

%Let $X$ be a real analytic manifold.  
From now on, the symbol $*$ denotes $(s)$ or $\{s\}$ for some $s > 1$.
Let us recall the def\mbox{}inition of the sheaf ${\ms C}^*$ of
ultra-dif\mbox{}ferentiable functions of class $*$ in $X$.

First we need the notion of $1$-regular sets. 
\begin{defs}%[See \cite{whitney_ext} and \cite{kurdyka}]
\label{defs_regular}
We say that $A$ is 1-regular at $p \in X$ if there exist a
neighborhood $U \subset X$ of $p$, a neighborhood $V \subset \mathbb
R^n$ of the origin and an isomorphism $\psi: (U,p) \to (V,0)$
satisfying the following condition. There exist a positive constant
$\kappa > 0$ and a compact neighborhood $K \subset V$ of the origin
such that for any $x_1,x_2\in\psi(A \cap U) \cap K$ there exists a
subanalytic curve $l$ in $\psi(A \cap U)$ joining $x_1$ and $x_2$ and
satisfying the estimate
$$
|l| \le \kappa \vert x_1 - x_2 \vert,
$$
where $|l|$ stands for the length of $l$.

The set $A$ is said to be 1-regular if it is 1-regular at any point $p
\in X$.
\end{defs}

For a locally closed subanalytic subset $A$
(see Definition {\ref{def:semi-sub-analytic}} for a subanalytic set),
if $A$ is 1-regular then,
using the Curve Selection Lemma (see \cite{ks_som}), one proves that
$\overline{A}$ is also 1-regular. Clearly, the converse does not hold. For
example, let $X=\rea^2$ with coordinates $(x,y)$, $A=X\setminus
\{x=0\}$ is not 1-regular at any point in the set $\{x=0\}$, but
$\overline{A} = \mathbb R^2$ is 1-regular at every point. Moreover a
1-regular set is locally connected at every point in $X$, that is, for any $p \in X$, there
exists a family $\{V_i\}$ of fundamental neighborhoods of $p$
satisfying $V_i \cap A$ is connected.  %But it might be not connected.

%Now, let $X=\rea^n$. Given $\alpha=(\alpha_1,\ldots,\alpha_n)\in\integer_{\geq0}^n$,
%we set $|\alpha|:=\sum_{j=1}^n\alpha_j$ and $D^\alpha:=\bigr{\frac
%  \partial{\partial x_1}}^{\alpha_1}\ldots\bigr{\frac
%  \partial{\partial x_n}}^{\alpha_n}$.

% N.Honda
% If $K$ has no interior point, the definition become meaningless.
% Moreover we need K = closure of int(K). Hence ..
%

\

Let $(x_1, x_2, \dots, x_n)$ be a system of coordinates of $\rea^n$ and
$V \subset \rea^n$ a 1-regular relatvely compact open subset.
Let us recall that, given $s>1$ and $h>0$, 
the space $\ms C^{s,h}(\overline{V})$ consists of $f\in\ms C^\infty(V)$
whose arbitrary partial derivative extends to a continuous function
on $\overline{V}$ with the following growth condition.  There exists $C>0$ such that, for any
$\alpha\in(\integer_{\geq0})^n$,
\begin{equation}
  \label{eq:gevrey_est}
  \ou{x\in \overline{V}}{\sup}\big|D^\alpha f (x) \big|\leq Ch^{|\alpha|}(|\alpha|!)^s \ ,
\end{equation}
where 
$D^\alpha:=\bigr{\frac\partial{\partial
    x_1}}^{\alpha_1}\ldots\bigr{\frac \partial{\partial
    x_n}}^{\alpha_n}$
for $\alpha=(\alpha_1,\ldots,\alpha_n)\in(\integer_{\geq0})^n$.

We denote by $\ms D^{s,h}_{\overline{V}}$ 
the set of
functions $f \in \ms C^{s,h}(\overline{V})$ 
with $D^\alpha f \vert_{\overline{V}\setminus V} = 0$ for any $\alpha \in(\integer_{\geq0})^n$.
% and satisfying \eqref{eq:gevrey_est} for
The spaces $\ms C^{s,h}(\overline{V})$ and $\ms D^{s,h}_{\overline{V}}$ endowed with the
norm
$$ ||f||_{\overline{V},s,h}:=\ou{
\begin{subarray}{c}x \in \overline{V}\\\alpha\in(\integer_{\geq0})^n\end{subarray}  
}{\sup}\ \frac{|D^\alpha f (x)|}{h^{|\alpha|}(|\alpha|!)^s}$$ 
are Banach spaces.

Given an open set $U\subset \rea^n$, we set
\begin{eqnarray*}
\ms C^{ (s) }(U)  &  :=  &  \ou{V\Subset U}{\varprojlim}\ou{h>0}{\varprojlim}\ms C^{s,h}(\overline{V}) \ , \\
\ms C^{\{s\}}(U)  &  :=  &  \ou{V\Subset U}{\varprojlim}\ou{h>0}{\varinjlim}\ms C^{s,h}(\overline{V})  \ .
\end{eqnarray*}
Here $V$ runs through 1-regular relatively compact open subsets of $U$.

In \cite{komatsu_ultradistributionsII} (see also
\cite{roumieu_ultradistributions}), it is proved that, given open sets
$W,W'\subset\rea^n$ and a real analytic isomorphism $\Phi:W\to W'$,
the morphism $\cdot\circ\Phi:\ms C^*(W')\to \ms C^*(W)$ is
an isomorphism. Hence, for
an open subset $U$ in a real analytic mamifold $X$,
the set of ultra-dif\mbox{}ferentiable functions
$\ms C^{*}(U)$ is also well-defined.

\

Let $X$ be a real analytic manifold.
One checks easily that $\ms C^{\{s\}}$ and $\ms C^{(s)}$ 
are sheaves on $X$. 

%For the rest of the paper $*=(s)$ or $\{s\}$.
%For the rest of the paper the index $*$ will mean $\{s\}$ or $(s)$.

\begin{defs}
  For a locally closed subset $A$ in $X$, the image sheaf $j_A({\ms
    C}^*) \subset {\mathcal J}_A$ is called the sheaf of Whitney jets
  of class $*$ over $A$, and we denote it by ${\mathcal W}_A^*$.
\end{defs}

Note that ${\mathcal W}^*_A$ is a sheaf of rings and a ${\mathcal
  D}_X$-module.  By the partition of the unity, the def\mbox{}inition
of ${\mathcal W}^*_A$ is equivalent to the following.  Given $F \in
{\mathcal J}_A(U)$, then $F$ belongs to ${\mathcal W}_A^*(U)$ if and
only if there exists $f \in {\ms C}^*(U)$ with $j_A(f) = F$.

It follows from the def\mbox{}inition that for any $F \in {\mathcal
  W}^*_A(U)$ there exists an unique $\tilde{F} \in {\mathcal
  W}^*_{\overline{A}}(U)$ such that $j_{A,\overline{A}}(\tilde{F}) = F$. Hence
the restriction map
$$
j_{A,\overline{A}}:  {\mathcal W}^*_{\overline{A}}(U) \to
{\mathcal W}^*_A(U)
$$
is an isomorphism.

Let $A \subset \rea^n$ be a locally closed set, $U \subset \rea^n$ an open set and $s > 1$.
We introduce two families of semi-norms on $\mc J_A(U)$.

For $h>0$ and $K \subset A \cap U$, set
$$ ||F||_{K,s,h}:= 
\ou
{\begin{subarray}{c}y \in K\\ \alpha \in (\integer_{\ge 0})^n\end{subarray}}
\sup
\frac{\vert f_\alpha(y) \vert}{\vert \alpha \vert!^sh^{\vert \alpha
    \vert}} 
\qquad(F = \{f_\alpha\}_{\alpha\in (\integer_{\ge0})^n} \in {\mathcal J}_A(U))
\ . $$

%we will def\mbox{}ine the function 
%$$
%\begin{array}{rrcl}
%  \vert\vert \cdot \vert\vert_{K,s,h}:  &  {\mathcal J}_A(U)  &  \lra  &  (\mathbb R_{\ge 0} \cup \{\infty\})\\
%  &  F(y)=\{f_\alpha(y)\}_{\alpha}  & \longmapsto  &  ||F||_{K,s,h}:= \ou{\begin{subarray}{c}y \in K\\ \alpha \in (\integer_{\ge 0})^n\end{subarray}}\sup
%  \frac{\vert f_\alpha(y) \vert}{\vert \alpha \vert!^sh^{\vert \alpha \vert}}.
%\end{array}
%$$

We define, for $h > 0$ and $K\subset A\cap U$,
another semi-norm
$\vert\vert \cdot \vert\vert_{K,s,h}^{\mathcal W}$ on $\mc J_A(U)$ 
in the following manner.
Given $F=\{f_\alpha\}_{\alpha \in (\integer_\ge0)^n} \in {\mathcal J}_A(U)$ and $\beta \in
(\integer_{\ge 0})^n$ and $x,x' \in A\cap U$, we set
\begin{eqnarray*}
S_\beta F  &  :=  &  \{f_{\alpha+\beta}\}_\alpha \ ,\\
T_m(F;\,x,\,x')   & :=  & \sum_{\vert \alpha \vert \le m}
\frac{1}{\alpha!}f_\alpha(x')(x-x')^{\alpha}  \ , \\
R_m(F;\,x,\,x')  & :=  & f_0(x) - T_m(F;\,x,\,x')  \ . \\ 
\end{eqnarray*}
Then $\vert\vert F \vert\vert_{K,s,h}^{\mathcal W}$ is defined by
$$
\sup_{m \ge 0, \alpha}
\frac{m!}{(\vert \alpha \vert + m+1)!^s h^{\vert \alpha \vert + m+1}}
\left(\sup_{x,x' \in K,\, x \ne x'} 
\frac{\vert R_m( S_\alpha F;\, x,\,x') \vert}{\vert x - x' \vert^{m+1}}
\right).
$$
% N.Honda: OK
%\footnote{why $x,x'\in K\cap A$ and not simply $x,x'\in K$?}
%
The following characterization of $\Swj$ can be found in
%\cite{whitney_ext} or 
 \cite{kantor}.

\begin{theorem}\label{thm:w_norm_charact}
Let $A \subset \rea^n$ be a locally closed set and $U\subset \rea^n$ an open set.
%Then a jet $F \in {\mathcal J}_A(U)$ belongs to $\Swj (U)$ if and only if
\begin{enumerate}
\item A jet $F \in {\mathcal J}_A(U)$ belongs to $\mc{W}^{(s)}_A (U)$
  if and only if, for any $h > 0$ and any compact set $K$ in $U$,
  $\vert\vert F \vert\vert_{A \cap K,s,h} < \infty$ and $\vert\vert F
  \vert\vert_{A \cap K,s,h}^{\mathcal W} < \infty$ hold.
\item A jet $F \in {\mathcal J}_A(U)$ belongs to $\mc{W}^{\{s\}}_A
  (U)$ if and only if, for any compact set $K$ in $U$, there exists
  $h>0$ such that $\vert\vert F \vert\vert_{A \cap K,s,h} < \infty$
  and $\vert\vert F \vert\vert_{A \cap K,s,h}^{\mathcal W} < \infty$.
\end{enumerate}
\end{theorem}

\

Let $A \subset X = \mathbb R^n$ be a compact set. We set
$$
{\mathcal W}^{s,h}_A(X) = \{ F \in {\mathcal J}_A(X);\,
\vert\vert F \vert\vert_{A,s,h} +
\vert\vert F \vert\vert_{A,s,h}^{\mathcal W} < \infty\} \ .
$$
Endowing ${\mathcal W}^{s,h}_A(X)$ with the norm $\vert\vert \cdot
\vert\vert_{A,s,h} + \vert\vert \cdot \vert\vert_{A,s,h}^{\mathcal
  W}$, it becomes a Banach space.  It follows from Theorem
\ref{thm:w_norm_charact} that
\begin{eqnarray*}
{\mathcal W}^{(s)}_A(X) &=& \ou{h>0}{\varprojlim}\, {\mathcal W}^{s,h}_A(X)\qquad\text{and}\\
{\mathcal W}^{\{s\}}_A(X) &=& \ou{h>0}{\varinjlim}\, {\mathcal W}^{s,h}_A(X) \ .
\end{eqnarray*}
It follows that $\Swj(X)$ can be endowed with a locally convex
topology by these projective or inductive limits.  It is easy to see
$\Swj(X)$ is an {\bf{FS}} space (resp. a {\bf{DFS}} space) if $* = (s)$
(resp.  $*=\{s\}$) respectively.

\subsection{Ultradistributions}

For a complete presentation of the theory of ultradistributions, we cite
\cite{komatsu_ultradistributionsI}.

Let $X$ be a real analytic manifold. Let us recall that, given a sheaf
$F$ on $X$ and $U\subset X$ an open set, we denote by $\Gamma_c(U,F)$,
the set of sections of $F$ on $U$ with compact support.  

Given $U\subset \rea^n$ open, the locally convex topological vector spaces $\ms D^{*}(U)$ and $\ms
D^{s,h}(U)$ are def\mbox{}ined as
\begin{eqnarray*}
  \ms D^{s,h}(U) & := & \ou{V\Subset U}{\varinjlim}\ \ms D^{s,h}_{\overline{V}}\simeq \Gamma_c(U,\ms C^{s,h}) \ ,\\
  \ms D^{(s)}(U) & := & \ou{V\Subset U}{\varinjlim}\ \ou{h>0}{\varprojlim}\ \ms D^{s,h}_{\overline{V}}\simeq \Gamma_c(U,\ms C^{(s)}) \ ,\\
  \ms D^{\{s\}}(U) & := & \ou{V\Subset U}{\varinjlim}\ \ou{h>0}{\varinjlim}\ \ms D^{s,h}_{\overline{V}}\simeq \Gamma_c(U,\ms C^{\{s\}}) \ .  
\end{eqnarray*}

In \cite{komatsu_ultradistributionsII} (see also
\cite{roumieu_ultradistributions}), it is proved that, given open sets
$W,W'\subset\rea^n$ and a real analytic isomorphism $\Phi:W\to W'$,
the morphism $\cdot\circ\Phi:\ms D^*(W')\to \ms D^*(W)$ is
an isomorphism. Hence we can define
$\ms D^{*}(U)$ for an open set $U$ in a real analytic manifold.

% N.Honda: Ultradistributions are not defined yet.
%Further the morphism $\cdot\circ\Phi$ induces an
%isomorphism between the space of ultradistributions of class $*$ on
%$W'$ and $W$.

\begin{defs}
  \begin{enumerate}
  \item Let $X$ be a real analytic manifold of dimension $n$. We
    denote by $\ms V^*$ the sheaf on $X$ of volume elements with
    coefficients in $\ms C^*$, that is $\ms V^*:= \ms
    C^*\underset{\mathcal A}{\otimes} \omega^{(n)}_X \underset{\integer}{\otimes} or_X$, 
    where $\omega^{(n)}_X$ $($resp. $\mathcal A)$ is the sheaf of real analytic n-forms 
    $($resp. functions$)$ on $X$, and $or_X$ is that of orientations on
    $X$.
\item Let $U$ be an open subset of $X$. The space of
  \emph{ultradistributions on $U$ of class $(s)$ of Beurling type}
  $($resp. \emph{of class $\{s\}$ of Roumieu type}$)$, or simply \emph{of
    class $(s)$} $($resp. \emph{$\{s\}$}$)$, denoted $\mc Db^{(s)}(U)$
  $($resp. $\mc Db^{\{s\}}(U))$, is def\mbox{}ined as the strong dual space of
  $\Gamma_c(U,\ms V^{(s)})$ $($resp. $\Gamma_c(U,\ms V^{\{s\}}))$.
  \end{enumerate}
\end{defs}

In \cite {pierre_ultradistributions}, it is proved that $\mc Db^{*}$
is a sheaf on $X$.

Given a closed set $Z\subset X$ and $F\in\Mod(\com_X)$ denote by
$\Gamma_Z(F)$ the subsheaf of $F$ of sections supported by $Z$.

If $A$ is a compact subset of $X=\rea^n$, from results of
H. Whitney and J.-M. Kantor, it follows that the topological dual of
$\mc W_A^{*}(X)$ is isomorphic to $\Gamma_A(X,\mc Db^{*})$.

\begin{defs}
  The presheaf of tempered ultradistributions of class $*$ on $X$,
  denoted $\mc Db^{*t}_X$, is def\mbox{}ined by
  $$\mc Db^{*t}_X(U):=\frac{\Gamma(X;\mc
    Db^{*})}{\Gamma_{X\setminus U}(X;\mc Db^{*})} \ ,\quad U\subset
  X\text{ open.}$$ %is called \emph{the presheaf of tempered
%    ultradistributions of class $*$}.
Note that $\mc Db^{*t}_X$ is not a sheaf on $X$.
\end{defs}

In Proposition \ref{prop_komatsu_lemma} below (originally due to H.
Komatsu), we recall a suf\mbox{}f\mbox{}icient condition for a continuous function
to extend to $\rea^n$ as an ultradistribution. 
% We denote by $\ms C^0$ the sheaf on $\rea^n$ of continuous
% functions.
For $U$ an open subset
of $\rea^n$, we denote by $\mc Db^{s,l}(U)$ the dual space of $\ms
D^{s,l}(U)$, and by $L_{loc}(U)$ the set of locally integrable functions on $U$.

\begin{prop}\label{prop_komatsu_lemma}
  Let $U\subset\rea^n$ be an open set and $f \in L_{loc}(U)$.
Suppose that there exist positive constants $h$ and $C$ satisfying
$$ |f(x)|\leq C\exp\bigg(\frac h{\mrm{dist}(x,\partial U)^{\frac1{s-1}}}\bigg)  \ .$$ 
Then we have $f\in\mc Db^{(s)t}_X(U)$.
\end{prop}

We need some technical results.

\begin{lemma}\label{lemma_komatsu_prelemma}
  Let $U$ be a relatively compact open subset of $\rea^n$.
  \begin{enumerate}
  \item Let $f \in L_{loc}(U)$. If there exist positive constants $l$ and $M$ such that, for any
  $\phi\in\ms D^{s,l}(U)$,
$$ \bigg|\int_Uf\,\phi\,dx\bigg|\leq M
  \vert\vert \varphi \vert\vert_{U,\, s,\, l}, $$
  then $f$ extends to an ultradistribution in $\mc Db^{s,l}(\rea^n)$. In particular,
  $f$ belongs to $\mc Db^{(s)t}_X(U)$.
  
\item There exist constants $C>0$ and $\kappa > 0$ such that, for any $l>0$ and for any
  $\phi\in\ms D^{s,l}(U)$, the inequality
$$ |\phi(x)|\leq C\exp\left({\frac{-\kappa l^{\frac{-1}{s-1}}}{\mathrm{dist}(x,\partial U)^{\frac1{s-1}}}}\right)
  \vert\vert \varphi \vert\vert_{U,\,s,\,l} 
$$
holds for any $x\in U$. 
  \end{enumerate}
\end{lemma}

\emph{Proof}. \emph{(i).} By the Hahn-Banach's extension Theorem the
functional $\int_Uf\cdot\,dx$ extends to $\ms D^{s,l}(\rea^n)$.

\emph{(ii).} Set
$$
 M := \ou
{\begin{subarray}{c} \alpha\in(\integer_{\geq0})^n\\x\in U
  \end{subarray}} {\sup}
\frac{|D^\alpha\phi(x)|}{(\vert\alpha\vert !)^sl^{|\alpha|}},
$$
and let $y$ be a point in $\partial U$ with $\vert x - y \vert =
\mathrm{dist}(x, \partial U)$. Then we have 
$$
\varphi(x) = \frac{1}{(k-1)!}\int_0^1 (1-t)^{k-1}\frac{d^k}{dt^k}\varphi ((x-y)t +y) dt
\qquad \text{for any $k \in \mathbb N$}.
$$
As
$$
\begin{aligned}
\frac{d^k}{dt^k}\varphi ((x-y)t +y) 
&= \left.\left( \sum_{i=1}^n (x_i - y_i)\frac{\partial}{\partial z_i}\right)^k \varphi(z)\right|_{z = (x-y)t+y} \\
&= k! \sum_{\vert \alpha \vert = k} \frac{(x-y)^\alpha}{\alpha !} 
\left. D_z^\alpha \varphi(z)\right|_{z = (x-y)t + y}
\end{aligned}
$$
hold, we obtain 
$$
\vert \varphi(x) \vert \le k(l\vert x-y \vert)^k (k!)^sM \sum_{\vert \alpha \vert = k}\frac{1}{\alpha !}
\le (2nl\vert x -y\vert)^k (k!)^{s-1} M.
$$
Hence we have
$$
\vert \varphi(x) \vert
\le  \inf_{k \ge 1} (2nl\vert x - y \vert)^k(k!)^{s-1} M
= \left(\inf_{k \ge 1} \epsilon^kk!\right)^{s-1} M
$$
where we set $\epsilon := (2nl\vert x-y \vert)^{\frac{1}{s-1}} =
(2nl\operatorname{dist}(x, \partial U))^{\frac{1}{s-1}}$.
We may assume $\epsilon \le 1$. Then, for an integer $j_0 \ge 1$ satisfying 
$ j_0 \le \frac{1}{\epsilon} < j_0+1$,
we get
$$
\begin{aligned}
\inf_{k \ge 1} \epsilon^kk! &\le \epsilon^{j_0} {j_0}! 
\le \left(\frac{1}{j_0}\right)^{j_0} {j_0}!
\le e^{\frac{1}{12}} \sqrt{j_0}e^{-j_0} \\
&\le e^{\frac{13}{12}} \frac{1}{\sqrt{\epsilon}}\exp\left(-\frac{1}{\epsilon}\right)
\le 2e^{\frac{13}{12}} \exp\left(-\frac{1}{2\epsilon}\right)
\end{aligned}
$$
thanks to the Stirling formula 
$$ 
j!=\sqrt{2\pi}j^{j+\frac12}e^{-j + \frac{\theta(j)}{12j}}
\qquad 0 < \theta(j) < 1.
$$
%Hence there exists constants $C > 0$ and $\kappa > 0$ independent of $\varphi$ and $l$ such that
%$$
%\vert \varphi(x) \vert \le
%C\exp\left({\frac{-\kappa}{(l\mathrm{dist}(x,\partial U))^{\frac1{s-1}}}}\right) M.
%$$
This completes the proof.
\proofend

\emph{Proof of Proposition \ref{prop_komatsu_lemma}.} 
Since the problem is local, we may assume that $U$ is relatively compact.
Lemma \ref{lemma_komatsu_prelemma} (ii) implies that there exists
$C'>0$ and $\kappa > 0$ such that, for any $l > 0$ and for any $\phi\in\ms D^{s,l}(U)$, $$
|f(x)\phi(x)|\leq C'\exp\biggr{\frac{h-\kappa l^{\frac{-1}{s-1}}}{\mrm{dist}(x,\partial
    U)^{\frac1{s-1}}}} 
\vert\vert \varphi \vert\vert_{\overline{U},s,l} \ .$$

If we take $l > 0$ sufficiently small, then
the conclusion follows from Lemma \ref{lemma_komatsu_prelemma} (i).

%Combining Lemma \ref{lem_komatsu_prelemma} 
%\emph{(i)} and \emph{(ii)}, the result follows.

\proofend

\subsection{An example} \label{section_example}

Let $l\geq1$. Set 
$$ U_1:=\{(x,y)\in\rea^2;\ y>x^{2l+1}\} \ , $$
$$ U_2:=\rea\times\rea_{<0} \ . $$
Def\mbox{}ine a function $u_i(x,y) \in L_{loc}(U_i)$ ($i=1,2$) by
\begin{eqnarray*}
  u_1(x,y)&:=&
  \begin{cases}
    \exp\left(\displaystyle\frac1{y-x^{2l+1}}\right) & \textrm{for } (x,y)\in
    U_1\cap(\rea\times\rea_{>0})\\
0 & \textrm{for } (x,y)\in
    U_1\cap(\rea\times\rea_{\leq0})
  \end{cases}
\\
u_2(x,y)&:=& 0 \quad \textrm{for any } (x,y)\in U_2
\end{eqnarray*}

By Proposition \ref{prop_komatsu_lemma}, $u_1\in\mc
Db^{(2)t}(U_1)$. Clearly $u_2\in\mc Db^{(2)t}(U_2)$. As $u_1|_{U_1\cap
U_2}=u_2|_{U_1\cap U_2}$, there exists $u \in L_{loc}(U_1\cup U_2)$ such
that $u|_{U_1}=u_1$ and $u|_{U_2}=u_2$, but $u\notin\mc
Db^{(2)t}(U_1\cup U_2)$
whose proof will be given later.

It follows that tempered distributions do no glue on f\mbox{}inite coverings
of open sets with smooth boundaries. In particular, the presheaf $\mc
Db^{*t}_X$ is not a sheaf on the subanalytic site relative to $\rea^2$
(see Subsection \ref{SUBSEC_RECALL_XSA} for the def\mbox{}inition of
subanalytic site).
%$\mc Db^{(s)t}$ is not a sheaf on the subanalytic
%site.

Among the purpose of this paper there is the attempt to overcome the
dif\mbox{}f\mbox{}iculty presented in this example. We will
def\mbox{}ine a subanalytic sheaf of stratif\mbox{}ied Whitney jets of
Gevrey order. Then, we will def\mbox{}ine stratif\mbox{}ied
ultradistributions. In the end, by means of stratif\mbox{}ied
ultradistributions, we will def\mbox{}ine tempered-stratif\mbox{}ied
ultradistributions and we will prove two results. The first states
that, if $X$ is a real surface, tempered-stratif\mbox{}ied
ultradistributions def\mbox{}ine a subanalytic sheaf. The second
states that the sections of tempered-stratif\mbox{}ied
ultradistributions on open subanalytic open sets with $1$-regular
complementary coincides with sections of tempered ultradistributions.

Now, let us prove that $u\notin\mc Db^{(2)t}(U_1\cup U_2)$. 

Set $U := U_1 \cup U_2$ and $D := \{(x,y)\in \rea^2;\, \vert x \vert < 1, \vert y \vert < 1\}$.
Suppose that  $u\in\mc Db^{(2)t}(U)$. Then there
exist positive constants $C$ and $h$ such that
$$
\left| \int_{U} u \psi dx \right| \le C \vert\vert \psi \vert\vert_{\overline{D},2,h}
$$
holds for any $\psi \in \ms D^{(2)}(U \cap D)$.
Now applying Lemma \ref{lemma:partition-of-unity} to
the situation $\varphi(t)= t^2$, $s=2$ and $K=\{0\} \subset \rea^2$,
we obtain a family of functions $\{\chi_\epsilon(x,y)\}_{\epsilon > 0}$ 
satisfying the conditions 1., 2. and 3. of the lemma.
Set
$$
D_\epsilon := \left\{-\frac{5\epsilon}{2} \le x \le -\frac{3\epsilon}{2},\,
0 \le y \le x^{2l+1}\right\}
$$
and
$$
\psi_\epsilon(x,y) := \frac{\chi_\epsilon(x+2\epsilon,y)}{C_h\exp(\varphi(\epsilon^{-1}))}
= \frac{\exp(-\epsilon^{-2})\chi_\epsilon(x+2\epsilon,y)}{C_h}
$$
where $C_h > 0$ is the positive constant given in Lemma
\ref{lemma:partition-of-unity} 1.  Remark that $\vert\vert
\psi_\epsilon \vert\vert_{\overline{D},2,h} \le 1$ holds for any
$\epsilon > 0$, and hence, $\left| \displaystyle\int_{U} u
  \psi_\epsilon dx \right|$ is uniformly bounded.  On the other hand,
for suf\mbox{}f\mbox{}iciently small $\epsilon > 0$, we have
$$
\psi_\epsilon\vert_{D_\epsilon} =
\frac{1}{C_h}\exp\left(-\epsilon^{-2}\right) \ ,
$$
and 
$$
\exp\left(\frac{1}{y-x^{2l+1}}\right)
\ge
\exp\left(\frac{1}{-2x^{2l+1}}\right)
\ge
\exp\left(\kappa \epsilon^{-(2l+1)}\right)
\qquad ((x,y) \in D_\epsilon)
$$
for some positive constant $\kappa > 0$. Therefore we have
$$
\int_{U} u \psi_\epsilon dx \ge
\int_{D_\epsilon} u \psi_\epsilon dx \ge
\frac{1}{C_h}\exp\left(-\epsilon^{-2} +\kappa \epsilon^{-(2l+1)}\right)
\int_{D_\epsilon} dx \to \infty,
$$
which give a contradiction. Hence we conclude that $u\notin\mc Db^{(2)t}(U_1\cup U_2)$.

\section{Stratif\mbox{}ied Whitney jets and stratif\mbox{}ied ultradistributions}\label{section:swj}
\subsection{Review on the subanalytic site}\label{SUBSEC_RECALL_XSA}

Let $X$ be a real analytic manifold countable at inf\mbox{}inity.
%$\mathcal A$ the sheaf of real-valued real analytic functions on $X$.

\begin{defs}{\label{def:semi-sub-analytic}}
\begin{enumerate}
\item A set $Z\subset X$ is said \emph{semi-analytic at} $x\in X$ if
  the following condition is satisf\mbox{}ied. There exists an open
  neighborhood $W$ of $x$ such that $Z\cap W=\cup_{i\in I}\cap_{j\in
    J}Z_{ij}$ where $I$ and $J$ are f\mbox{}inite sets and either
  $Z_{ij}=\{y\in X;\ f_{ij}(y)>0\}$ or $Z_{ij}=\{y\in X;\
  f_{ij}(y)=0\}$ for some real-valued real analytic functions $f_{ij}$ on $W$. Further, $Z$ is
  said \emph{semi-analytic} if $Z$ is semi-analytic at any $x\in X$.
\item A set $Z\subset X$ is said \emph{subanalytic} if the following
  condition is satisf\mbox{}ied. For any $x\in X$, there exist an open
  neighborhood $W$ of $x$, a real analytic manifold $Y$ and a
  relatively compact semi-analytic set $A\subset X\times Y$ such that
  $\pi(A)=Z\cap W$, where $\pi: X\times Y\to X$ is the projection.
\end{enumerate}
\end{defs}

Given $Z\subset X$, denote by $\mathring Z$ (resp. $\overline
Z$, $\partial Z$) the interior (resp. the closure, the boundary) of $Z$. 

\begin{prop}[See \cite{bm_ihes}]
Let $Z$ and $V$ be subanalytic subset of $X$. Then $Z\cup V$, $Z\cap
V$, $\overline Z$, $\mathring Z$ and $Z\setminus V$ are
subanalytic. Moreover the connected components of $Z$ are subanalytic,
the family of connected components of $Z$ is locally f\mbox{}inite and $Z$ is
locally connected at any point in $Z$.
\end{prop}
% N.Honda
% The last assertion is false without ``at any point in $Z$''.
% Consider $Z=\rel^2\setminus {y=0,x\ge0}$
%In the litterature (\cite{whitney_4pages},
%\cite{komatsu_ultradistributionsI} \cite{kurdyka}), a $1$-regular set
%is also said to have \emph{the Whitney property with exponent $1$} or
%the \emph{property P}.

\begin{defs}
  \begin{enumerate}
  \item A family $\{A_\alpha\}_{\alpha\in\Lambda}$ of subanalytic
    subsets of $X$ is said a \emph{stratif\mbox{}ication of $X$} if
    $\{A_\alpha\}_{\alpha\in\Lambda}$ is locally f\mbox{}inite,
    $X=\ou{\alpha\in\Lambda}{\bigsqcup}A_\alpha$ and each $A_\alpha$
    is a locally closed subanalytic manifold.
  \item Given a locally closed subanalytic set $A\subset X$, we say
    that $\{A_\alpha\}_{\alpha\in\Lambda}$ is a \emph{stratif\mbox{}ication of $A$}
    if $A$ is the disjoint union of the $A_\alpha$ and there exists a
    stratif\mbox{}ication $\{A_\alpha\}_{\alpha\in\Lambda\cup\Lambda'}$ of
    $X$ f\mbox{}iner than the stratif\mbox{}ication $\{A,X\setminus A\}$.
  \item Let $A$ be a subanalytic subset of $X$,
    $\{A_\alpha\}_{\alpha\in\Lambda}$ a stratif\mbox{}ication of $A$. Then
    $\{A_\alpha\}_{\alpha\in\Lambda}$ is called a \emph{$1$-regular
      straf\mbox{}ication} if each stratum is $1$-regular, connected and
    relatively compact.
  \end{enumerate}
%  Let $Z$ be a subanalytic subset of $\rea^n$. A locally f\mbox{}inite family of subsets of
%  $\rea^n$, $\mc T$, is said \emph{a subanalytic stratif\mbox{}ication of
%    $\rea^n$ adapted to $Z$}, or simply \emph{a stratif\mbox{}ication adapted
%    to $Z$}, if
%  \begin{enumerate}
%  \item $\rea^n$ is the disjoint union of the elements of $\mc T$,
%  \item $Z$ is the disjoint union of the elements of a subfamily of $\mc T$,
%  \item the elements of $\mc T$ are subanalytic subsets and
%    subanalytic submanifolds of $\rea^n$,
%  \item for any $Z_1,Z_2\in\mc T$, if $Z_1\cap\overline
%    Z_2\neq\varnothing$, then $Z_1\subset\overline Z_2$ and $\dim
%    Z_1<\dim Z_2$.
%  \end{enumerate}
\end{defs}

\begin{prop}[See \cite{kurdyka}]\label{prop:kurdyka}
  \begin{enumerate}
  \item Let $Z\subset X$ be a locally closed subanalytic subset. There exists a
    $1$-regular stratif\mbox{}ication of $Z$.
  \item Let $U\subset X$ be a subanalytic open set. There exists a
    locally finite open covering $\{U_j\}_{j\in J}$ of $U$ such that,
    for any $j\in J$, $U_j$ is a subanalytic $1$-regular set.
\end{enumerate}
\end{prop}

%\begin{defs}
%  Let $\gamma:]-1,1[\longrightarrow M$, be an analytic map, $\ep>0$
%  such that $\gamma_{]0,\ep[}$ is bijective. The set $\gamma(]0,\ep[)$
%  is said \emph{a semi-analytic arc with an endpoint at $\gamma(0)$},
%  or simply \emph{a semi-analytic arc}. 
%\end{defs}

For the rest of the subsection we refer to \cite{ks_indsheaves}.

We denote by $\Op(X)$ the family of open subsets of $X$. For $k$ a
commutative ring, we denote by $\Mod(k_X)$ the category of sheaves of
$k$-modules on $X$.

Let us recall the def\mbox{}inition of the subanalytic site $\xsa$
associated to $X$. An element $U\in\Op(X)$ is an open set for $\xsa$
if it is open, relatively compact and subanalytic. The family of open
sets of $\xsa$ is denoted $\opxsac$. For $U\in\opxsac$, a subset $S$
of the family of open subsets of $U$ is said an open covering of $U$
in $\xsa$ if $S\subset\opxsac$ and, for any compact $K$ of $X$, there
exists a f\mbox{}inite subset $S_0\subset S$ such that
$K\cap(\cup_{V\in S_0}V)=K\cap U$. The set of coverings of $U$ in
$\xsa$ is denoted by $\cov U$.

We denote by $\Mod(k_\xsa)$ the category of sheaves of $k$-modules on
the subanalytic site. With the aim of def\mbox{}ining the category
$\Mod(k_\xsa)$, the adjective ``relatively compact'' can be omitted in
the def\mbox{}inition above. Indeed, in \cite[Remark 6.3.6]{ks_indsheaves},
it is proved that $\Mod(k_\xsa)$ is equivalent to the category of
sheaves on the site whose open sets are the open subanalytic subsets
of $X$ and whose coverings are the same as $\xsa$.

%Given $Y\in\opxsac$, we denote by $Y_{\xsa}$ the site induced by
%$\xsa$ on $Y$, def\mbox{}ined as follows. The open sets of $Y_{\xsa}$
%are open subanalytic subsets of $Y$. A covering of $U\in\Op(Y_{sa})$
%for the topology $Y_{\xsa}$ is a covering of $U$ in $\xsa$.

Let $\mrm{PSh}(k_\xsa)$ be the category of presheaves of
$k$-modules on $\xsa$.
Denote by $for:\Mod(k_\xsa)\to \mrm{PSh}(k_\xsa)$ the forgetful
functor which associates to a sheaf $F$ on $\xsa$ its underlying
presheaf. It is well known that $for$ admits a left adjoint
$\cdot^a:\mrm{PSh}(k_\xsa)\to \Mod(k_\xsa)$.

For $F\in \mrm{PSh}(k_\xsa)$, let us brief\mbox{}ly
recall the construction of $F^a$. 

For $U\in\opxsac$ and $S\in\cov U$, let $F(S)$ be def\mbox{}ined as the
kernel of the morphism
\begin{eqnarray*}
  \ou{U\in S}{\prod}F(U) & \lra  & \ou{U,V\in S}{\prod}F(U\cap V)\\
\{s_U\}_{U\in S}  &  \longmapsto  &  \{s_U|_{U\cap V}-s_V|_{U\cap V}\}_{U,V\in S} \ .
\end{eqnarray*}

If $S'\in\cov U$ is a ref\mbox{}inement of $S$, then there exists a
natural morphism $F(S)\lra F(S')$.

%set
%\begin{equation}\label{eq_F(S)}
%  F(S):=\Big\{(s_1,\ldots,s_n)\in\prod_{j=1}^nF(U_j);\ s_j|_{U_j\cap U_k}=s_k|_{U_j\cap U_k},j,k=1,\ldots,n\Big\} \ . \end{equation}
%If $S$ is a covering of $U$ and $S'$ is a refinement of $S$, then there exists a natural restriction morphism $F(S)\underset{\rho_{SS'}}{\lra}F(S')$. 

Now, for $U\in\opxsac$, set
\begin{equation}\label{eq_F+} 
F^+(U):=\underset{S\in\cov U}{\varinjlim} F(S) \ .
\end{equation}
It turns out that $F^a\simeq F^{++}$.

The following Lemma is an immediate consequence of the
def\mbox{}intions above.

\begin{lemma}\label{lemma:sheafification}
  Let $\mc T\subset\Op(\xsa)$ be such that for any $U\in\Op(\xsa)$ there
  exists $S\in\cov U$, such that $S\subset\mc T$. Let
  $F,G\in\mrm{PSh}(k_\xsa)$ and suppose that there exists a morphism
  of presheaves $\phi:F\to G$ such that, for any $V\in\mc T$,
  $\phi_V:F(V)\overset{\sim}{\to} G(V)$. Then $F^a\simeq G^a$.
\end{lemma}

We denote by
$$ \rho:X\lra\xsa \ ,$$ 
the natural morphism of sites associated to
$\opxsac\lra\Op(X)$. We refer to \cite{ks_indsheaves} for the
def\mbox{}initions of the functors $\rho_*:\Mod(k_X)\lra\Mod(k_\xsa)$
and $\rho^{-1}:\Mod(k_\xsa)\lra\Mod(k_X)$ and for Proposition
\ref{prop_functors} below.
\begin{prop}\label{prop_functors}
\begin{enumerate}
\item The functor $\rho^{-1}$ is left adjoint to $\rho_*$.
\item The functor $\rho^{-1}$ has a left adjoint denoted by
  $\rho_!:\Mod(k_X)\to\Mod(k_\xsa)$.
\item The functors $\rho^{-1}$ and $\rho_!$ are exact, $\rho_*$ is
  exact on constructible sheaves.
\item The functors $\rho_*$ and $\rho_!$ are fully faithful.
\end{enumerate}
\end{prop}

Through $\rho_*$, we will consider $\Mod(k_X)$ as a subcategory of $\Mod(k_\xsa)$.

The functor $\rho_!$ is described as follows. If
$U\in\opxsac$ and $F\in\Mod(k_X)$, then $\rho_!(F)$ is the sheaf on $\xsa$
associated to the presheaf $U\mapsto F\bigr{\overline U}$.

\subsection{Stratif\mbox{}ied Whitney jets}

%Let $A$ be a locally closed subanalytic subset in $X$.  
%From now on,
%the word ``stratif\mbox{}ication '' always means a subanalytic
%stratif\mbox{}ication, that is, each stratum is a locally closed
%subanalytic submanifold.  Moreover for a locally closed subanalytic
%subset $A$ in $X$, the stratif\mbox{}ication of $A$ means the set of
%strata contained in $A$ of a stratif\mbox{}ication f\mbox{}iner than
%$\{A,\, X\setminus A\}$, that is, it is the set $\{A_\alpha; A_\alpha
%\subset A\}$ where $\{A_\alpha\}$ is some subanalytic
%stratif\mbox{}ication of $X$ f\mbox{}iner than $\{A,\, X \setminus
%A\}$.

%\

Let $A$ be a locally closed subanalytic subset
in a real analytic manifold $X$ and $U \subset X$ an open subset.
Let us def\mbox{}ine the sheaf $\Sswj$ of stratif\mbox{}ied Whitney
jets of class $*$ over $A$. 
\begin{defs}{\label{def:stratified-Whitney-jets}}
  \begin{enumerate}
  \item We say that $F \in {\mathcal J}_A(U)$ is a stratif\mbox{}ied Whitney
    jet of class $*$ over $A$ in $U$ if for any compact subanalytic
    set $K$ in $U$ there exists a subanalytic stratif\mbox{}ication
    $\{A_\alpha\}_{\alpha \in \Lambda}$ of $A$ such that $j_{A_\alpha
      \cap K,A}(F) \in {\mathcal W}^*_{A_\alpha \cap K}(U)$ holds for
    any $\alpha \in \Lambda$.
  \item We denote by $\Sswj(U)$ the set of stratified Whitney jets of class $*$ over $A$ in $U$.
  \end{enumerate}
\end{defs}
\noindent

It is easy to verify that $\Sswj$ is a sheaf on $X$.

\noindent
\begin{exa}

Later we will prove that, if the set $A$ is 1-regular, then
    $\Swj = \Sswj$. However, in general, $\Swj$ and $\Sswj$ are different. For
    example, let $m \ge 2$, $X = \mathbb R^2$ with coordinates $(x,y)$ and
$$
B = \{(x,y) \in \mathbb R^2;\, y=0,\, x \ge 0\}\ ,\qquad
B_m = \{(x,y) \in \mathbb R^2;\, y=x^m,\, x\ge 0\}\ .
$$
Set $A = B \cup B_m$. We def\mbox{}ine the jet $F = \{f_\alpha\} \in
{\mathcal J}_A(X)$ by:
$$
f_\alpha(x,y) =
\left\{
\begin{array}{ll}
0 \qquad & (x,y) \in B_m \\
\\
\displaystyle\frac{\partial^\alpha}{\partial x^\alpha} 
\exp \left(-\displaystyle\frac{1}{x}\right)\qquad
& (x,y) \in B
\end{array}
\right. \ .
$$
Then $F \in {\mathcal S}{\mathcal W}^{\{2\}}_A(X)$, but
$F \notin {\mathcal W}^{\{2\}}_A(X)$. 
As a matter of fact, if $F \in {\mathcal W}^{\{2\}}_A(X)$, then we can find
$\tilde{F} = \{\tilde{f}_\alpha\}_{\alpha} \in {\mathcal S}{\mathcal W}^{\{2\}}_{D}(X)$ with 
$j_{\partial D, D}(\tilde{F}) = F$ where 
\begin{equation}{\label{eq:whitney-example-1}}
D=\{(x,y)\in\rea^2;\,0 \le x,\, 0 \le y \le x^m\}\ .
\end{equation}
By applying Lemma \ref{prop:decay-estimate} to $\tilde{F}$ and $D$ with the 1-regular stratification
$\{D\setminus \partial D,\, B\setminus\{0\},\, B_m\setminus\{0\},\, \{0\}\}$, we have
constants $C,l>0$ satisfying
$$
\vert f_0(x,0) \vert = \vert \tilde{f}_0(x,0) \vert \le
C\exp\left(-\frac{l}{x^m}\right)\qquad (x > 0),
$$
which is impossible.

We can also give the similar example on an open subanalytic set: Set $U := X \setminus D$ 
($D$ was given by {\eqref{eq:whitney-example-1}}),
and define the jet $G = \{g_\alpha\} \in {\mathcal J}_{U}(X)$ by
$$
g_\alpha(x,y) =
\left\{
\begin{array}{ll}
0 \qquad & (x,y) \in U \setminus \{x > 0,\, y < 0\} \\
\\
\displaystyle\frac{\partial^\alpha}{\partial x^\alpha} 
\exp \left(-\displaystyle\frac{1}{x}\right)\qquad
& (x,y) \in \{x > 0, \ y < 0\}
\end{array}
\right. \ .
$$
Then $G \in {\mathcal S}{\mathcal W}^{\{2\}}_{U}(X)$, but
$G \notin {\mathcal W}^{\{2\}}_{U}(X)$. The reason is the same
as that for the first example.
\end{exa}

%{\bf{Example}}\par
%\noindent

%\
\begin{rem}
\begin{enumerate}
\item
We cannot expect $ {\mathcal S}{\mathcal W}^*_{\overline{A}}(X) =
  {\mathcal S}{\mathcal W}^*_A(X) $ on the contrary to 
${\mathcal W}^*_{\overline{A}}(X) =
{\mathcal W}^*_{A}(X)$.  
For example, consider the
  case $X=\mathbb R^2$ with coordinates $(x,y)$ and $A=X\setminus
  \{x=0\}$.
%, then ${\mathcal S}{\mathcal W}^*_A(X) \ne {\mathcal
%    S}{\mathcal W}^*_{\overline{A}}(X)$.  
\item We have the equivalence, for every $F \in {\mathcal J}_A(U)$, 
$$
j_{A_\alpha \cap K, A}(F) \in {\mathcal W}^*_{A_\alpha \cap K}(U)
\iff
j_{\operatorname{clos}_A({A_\alpha \cap K}),\, A}(F) 
\in {\mathcal W}^*_{\operatorname{clos}_A({A_\alpha \cap K})}(U)
$$
where $\operatorname{clos}_A(B)$ denotes the closure of the set $B$ in
$A$. Hence, in Definition {\ref{def:stratified-Whitney-jets}},
the condition 
$$
\text{`` $j_{A_\alpha \cap K,A}(F) \in {\mathcal W}^*_{A_\alpha \cap K}(U)$ 
for any $\alpha \in \Lambda$''}
$$
can be replaced with 
$$
\text{`` $j_{\operatorname{clos}_A{(A_\alpha \cap K)}, A}(F) \in 
{\mathcal W}^*_{\operatorname{clos}_A{(A_\alpha \cap K)}}(U)$ for any $\alpha \in \Lambda$''}.
$$

In particular, if $A$ is a compact subanalytic subset in $X$, then $F
\in \Sswj(X)$ if and only if there exists a stratif\mbox{}ication
$\{A_\alpha\}_{\alpha \in \Lambda}$ such that $j_{\overline{A}_\alpha,
  A}(F) \in {\mathcal W}^*_{\overline{A}_\alpha}(X)$.
\end{enumerate}
\end{rem}

The sheaf $\Swj$ is a subsheaf of $\Sswj$, and $\Sswj$ is a sheaf of
rings and a ${\mathcal D}_X$-module.  Further, $\Swj$ and $\Sswj$ are
$\ms C^*$-modules and soft sheaves.  If $\{A_\alpha\}_{\alpha \in
  \Lambda}$ is a stratif\mbox{}ication of $A$, then we denote by
$\Ssawj$ the subsheaf of $\Sswj$ def\mbox{}ined by:
$$
\Ssawj(U) = \bigc{F \in {\mathcal J}_A(U);\, 
j_{A_\alpha,A}(F) \in {\mathcal W}^*_{A_\alpha}(U)
\,\text{ for any } \alpha \in \Lambda} \ .
$$

%The following lemma follows easily from the def\mbox{}inition of $\Sswj$.
%As it will be essential in what follows, 
%%However it is very important when we will construct the subanalytic sheaf 
%%${\mathcal C}^*_{X_{sa}}$, 
%we give the proof here.

\begin{lemma}{\label{lemma:stratified-Whitney-exact}}
  Let $A$ be a locally closed subanalytic subset in $X$, and let
  $\{A_i\}_{i=1}^l$ be a f\mbox{}inite family of locally closed
  subanalytic subset in $X$ with $A = \cup A_i$. We assume that every
  $A_i$ is a closed subset in $A$, or that every $A_i$ is open in $A$.
  Then the sequence of sheaves
\begin{equation}{\label{eq:exact-Whitney}}
0 \to {\mathcal S}{\mathcal W}^*_{A} \to
\underset{1 \le i \le k} \oplus {\mathcal S}{\mathcal W}^*_{A_i}
\to
\underset{1 \le i < j \le k} \oplus {\mathcal S}{\mathcal W}^*_{A_i \cap A_j}
\end{equation}
is exact.
\end{lemma}

\begin{proof}
  The injectivity of the second morphism of (\ref{eq:exact-Whitney})
  is clear.

Under the condition that every $A_i$ is closed (or open) in $A$, the
sequence of sheaves of jets
$$
0 \to {\mathcal J}_{A} \to
\underset{1 \le i \le k} \oplus {\mathcal J}_{A_i}
\to
\underset{1 \le i < j \le k} \oplus {\mathcal J}_{A_i \cap A_j}
$$
is exact. Let $U$ be an open subset and $K$ a compact subanalytic
subset in $U$, and let $F_i \in {\mathcal S}{\mathcal W}^*_{A_i}(U)$
with $j_{A_i \cap A_j, A_i}(F_i) = j_{A_i \cap A_j, A_j}(F_j)$.  Then
by the above exact sequence we can f\mbox{}ind a jet $F \in {\mathcal
  J}_A(U)$ with $j_{A_i, A}(F) = F_i$. To conclude the proof let us
show that $F \in {\mathcal S}{\mathcal W}^*_{A}(U)$.

As $F_i\in\Sswj_{A_i}(U)$ there exists a stratif\mbox{}ication
$\{A^{(i)}_\alpha\}_{\alpha}$ of $A_i$ such that $j_{A^{(i)}_\alpha
  \cap K, A_i} (F_i) \in {\mathcal W}^*_{A^{(i)}_\alpha \cap K}(U)$.
If we take a stratif\mbox{}ication $\{A_\alpha\}$ of $A$ f\mbox{}iner
than any partition $\{\{A^{(i)}_\alpha\}_\alpha,\, A \setminus A_i\}$
of $A$, then we have $j_{A_\alpha \cap K, A_i}(F_i) \in {\mathcal
  W}^*_{A_\alpha \cap K}(U)$ for any stratum $A_\alpha$ with $A_\alpha
\subset A_i$. Hence, for each $A_\alpha$ we conclude
$$
j_{A_\alpha \cap K, A}(F) =
j_{A_\alpha \cap K, A_i}(F_i) \in {\mathcal W}^*_{A_\alpha \cap K}(U)
$$
where the index $i$ is taken so that $A_\alpha \subset A_i$.
\end{proof}

Remark that, in general, the sequence
$$
0 \to {\mathcal W}^*_{A} \to
\underset{1 \le i \le k} \oplus {\mathcal W}^*_{A_i}
\to
\underset{1 \le i < j \le k} \oplus {\mathcal W}^*_{A_i \cap A_j}
$$
is not exact.

The following lemma is fundamental.

\begin{lemma}{\label{lemma:Whitney-curve}}
Let $X=\mathbb R^n$ and $A$ a locally closed subanalytic subset in $X$.
For any $F=\{f_\alpha\} \in \Sswj(X)$ and
any subanalytic curve $l\subset A$ joining  $x,x'\in A$, we have
\begin{equation}{\label{eq:1-regular-estimate}}
\vert R_m(F;\,x,\,x') \vert \le 
\frac{(\sqrt{n}\vert l \vert)^{m+1}}{m!} 
\sup_{\begin{subarray}{c}\vert \alpha\vert=m+1\\y \in l\end{subarray}} \vert f_\alpha(y) \vert
\end{equation}
where $\vert l \vert$ denotes the length of the curve $l$.
\end{lemma}

\begin{proof}
%{\bf{This lemma will be proved by the completely same technique used by Whitney, 
%the following is the reconf\mbox{}irmation just for our convenience.}}
  We recall the following formula of \cite{whitney_4pages}. Let
  $l\subset A$ be a subanalytic curve, for any $x_1,x_2,x_3\in l$
$$
T_m(F;x_1,x_2) - T_m(F;x_1,x_3) =
\sum_{\vert\beta\vert \le m}
\frac{R_{m-\vert \beta\vert}(S_\beta F; x_2, x_3)}{\beta!}(x_1 - x_2)^\beta \ .
$$
Noticing that
$$
R_{m - \vert \beta \vert} (S_\beta F; x_2, x_3) 
= \sum_{\vert \gamma \vert = m - \vert \beta \vert + 1}
\frac{f_{\beta+\gamma}(x_3)}{\gamma!}(x_2 - x_3)^\gamma 
+ R_{m - \vert \beta \vert + 1} (S_\beta F; x_2, x_3) \ ,
$$
we have
$$
\begin{aligned}
\vert T_m(F;x_1,x_2) - T_m(F;x_1,x_3) \vert &\le
\sum_{\vert\beta\vert \le m}
\sum_{\vert \gamma \vert = m - \vert \beta \vert + 1}
\left|
\frac{f_{\beta+\gamma}(x_3)}{\beta!\gamma!}(x_1 - x_2)^\beta(x_2 - x_3)^\gamma 
\right| \\
&+
\sum_{\vert\beta\vert \le m}\left|
\frac{R_{m-\vert \beta\vert+1}(S_\beta F; x_2, x_3)}{\beta!}(x_1 - x_2)^\beta\right|.
\end{aligned}
$$
The f\mbox{}irst term is estimated as follows. Suppose that
$x_1,x_2,x_3$ are in a sequential order along $l$, then we have
$$
\begin{aligned}
&\sum_{\vert\beta\vert \le m}
\sum_{\vert \gamma \vert = m - \vert \beta \vert + 1}
\left|
\frac{f_{\beta+\gamma}(x_3)}{\beta!\gamma!}(x_1 - x_2)^\beta(x_2 - x_3)^\gamma 
\right| \\
&\le
\left(\sup_{y \in l, \vert\alpha\vert = m+1} \vert f_\alpha(y)\vert\right)
\left(\sum_{k=1}^n \vert (x_2 - x_3)_k \vert\right)
\sum_{\vert\beta+\gamma \vert = m}
\left| \frac{1}{\beta!\gamma!}(x_1 - x_2)^\beta(x_2 - x_3)^\gamma 
\right| \\
&\le
\left(\sup_{y \in l, \vert\alpha\vert = m+1} \vert f_\alpha(y)\vert\right)
\sqrt{n}\vert x_2 - x_3 \vert \frac{(\sqrt{n})^m}{m!}
(\vert x_1 - x_2 \vert + \vert x_2 - x_3 \vert)^m \\
&\le
\frac{(\sqrt{n})^{m+1}\vert l \vert^m}{m!}\left(\sup_{y \in l, \vert\alpha\vert = m+1} 
\vert f_\alpha(y)\vert\right)\vert x_2 - x_3 \vert \ .
\end{aligned}
$$
Since $l$ is compact, we may assume that $A$ is compact.
%, for example,
%we set $A:= l$. 
%\marginpar{\tiny {why $l=A$? is it enough $A$ compact? in the
%  rest of the proof we don't use $l=A$}}  
Then, by the
def\mbox{}inition of $\Sswj$, there exists a stratif\mbox{}ication
$\{A_\tau\}_\tau$ of $A$ such that for any $\tau$ we have
$j_{\overline{A}_\tau, A}(F) \in {\mathcal W}^*_{\overline{A}_\tau}(X)$.  It
follows from Theorem \ref{thm:w_norm_charact} that, for any $\tau$,
there exists a constant $C_\tau$ such that
$$
\vert R_{m-\vert \beta \vert + 1}(S_\beta F;x_1, x_2)  \vert
\le C_\tau \vert x_1 - x_2 \vert^{m-\vert \beta \vert +2}
$$
holds for any $0 \le \vert \beta \vert \le m$ and $x_1$ $x_2 \in \overline{A}_\tau$.
As the number of strata is f\mbox{}inite, it makes sense to set
$$
C := \max_{\tau} C_\tau \ .
$$
Remark that the constant $C$ depends on $l$, $m$, $F$ and $A$ and it
does not depend on $x_1$ and $x_2$.  Now, the second term is
estimated in the following way. If $x_2,x_3\in l$ belong to the
closure of a same stratum, then
%$$
%\begin{aligned}
$$
\sum_{\vert\beta\vert \le m}\left|
\frac{R_{m-\vert \beta\vert+1}(S_\beta F; x_2, x_3)}{\beta!}(x_1 - x_2)^\beta\right| 
\le  
C \sum_{\vert\beta\vert \le m} \vert x_2 - x_3 \vert^{m-\vert \beta\vert + 2}
\frac{\vert x_1 - x_2 \vert^{\vert\beta\vert}}{\beta!} $$
\vspace{-7mm}
\begin{eqnarray*}
\phantom{aasdfasdfasdfadsfsldkfjasd}&\le &
C \vert x_2 - x_3 \vert^2 
\sum_{\vert\beta\vert \le m} \vert x_2 - x_3 \vert^{m-\vert \beta\vert}
\frac{\vert x_1 - x_2 \vert^{\vert\beta\vert}}{\beta!} \\
&\le &
C \vert x_2 - x_3 \vert^2 \vert l \vert^m \sum \frac{1}{\beta!}
\\
&\le & e^nC\vert l \vert^m \vert x_2 - x_3 \vert^2 \ .
%\end{aligned}
%$$
\end{eqnarray*}

Now we take points $x=x_0,x_1,\dots,x_k=x'$ sequentially in the curve $l$
so that each pair $x_i$ and $x_{i+1}$ belong 
to the closure of a same stratum ($0\le i \le k-1$).
Then
$$
\begin{aligned}
&\vert f_0(x) - T_m(F;x,x')\vert \\
&\qquad \le\vert T_m(F;x_0,x_0) - T_m(F;x_0,x_1) \vert +
\vert T_m(F;x_0,x_1) - T_m(F;x_0,x_2) \vert \\
&\qquad\qquad + \dots + \vert T_m(F;x_0,x_{k-1}) - T_m(F;x_0,x_k) \vert \\
&\qquad\le
\frac{(\sqrt{n})^{m+1}\vert l \vert^m}{m!}\left(\sup_{y \in l, \vert\alpha\vert = m+1} 
\vert f_\alpha(y)\vert\right) \sum_{i=0}^{k-1}\vert x_i - x_{i+1} \vert \\
&\qquad\qquad + e^nC\vert l \vert^m \sup_{0\le i \le k-1} \vert x_i - x_{i+1} \vert 
\sum_{i=0}^{k-1} \vert x_i - x_{i+1} \vert.
\end{aligned}
$$
When $k$ tends to $\infty$, then the f\mbox{}irst term in the right hand side converges to
$$
\frac{(\sqrt{n}\vert l \vert)^{m+1}}{m!}
\left(\sup_{y \in l, \vert\alpha\vert = m+1} 
\vert f_\alpha(y)\vert\right)
$$
and the second term tends to $0$. 

The conclusion follows.
\end{proof}

%Let $X$ be a real analytic manifold and $A$ a locally closed subanalytic subset in $X$.

\begin{coro}{\label{coro:1-regular-basic}}
  Let $X=\mathbb R^n$ and $A$ a locally closed 1-regular subanalytic
  subset in $X$.  Then for any subanalytic open set $V\subset X$ and any
  compact subanalytic set $K\subset V$ there exists $\kappa >
  0$ such that, for any $h > 0$, there exists $C_h>0$ satisfying
$$
\vert\vert F \vert\vert^{\mathcal W}_{K \cap A,\, s,\,\kappa h} 
\le C_h \vert\vert F \vert\vert_{\overline{V} \cap A,\,s,\,h} 
$$
for any $F \in \Sswj(X)$.
\end{coro}

\begin{proof}
  By the def\mbox{}inition of 1-regular, there exist a constant $M >
  0$ and a f\mbox{}inite family $\{V_i\}$ of open subsets in $X$ such
  that $K \subset \cup V_i\subset V$ and, for any $i$ and $x_1,x_2\in
  V_i \cap A$, there exists a curve $l\subset A \cap V$ joining $x_1$
  and $x_2$ satisfying $\vert l \vert \le M\vert x_1 - x_2 \vert$.

%\begin{enumerate}
%\item $K \subset \cup V_i\subset V$,
%\item for any $i$ and $x_1,x_2\in V_i \cap A$, there exists a 
%path $l\subset A \cap V$ joining $x_1$ and $x_2$ such that
%$$
%\vert l \vert \le M\vert x_1 - x_2 \vert.
%$$
%\end{enumerate}
Then, there exists a positive constant $\delta > 0$ such that for any
$x,y\in K$ with $\vert x - y \vert < \delta$, there exists $i$ such
that $x,y\in V_i$.

\

\noindent
F\mbox{}irst, assume that $x_1,x_2\in K \cap A$ satisfy $\vert x_1 - x_2
\vert < \delta$. Then, there exists a path $l\subset K\cap A$ joining
$x_1$ and $x_2$ such that $|l|\le M |x_1 - x_2|$. Then, Lemma
\ref{lemma:Whitney-curve} implies that
$$
\begin{aligned}
\vert R_m(S_\beta F;\,x_1,x_2) \vert &\le 
\frac{(\sqrt{n}\vert l \vert)^{m+1}}{m!} \max_{\vert\alpha\vert = \vert \beta \vert + m+1}
\sup_{y \in l} \vert f_\alpha(y) \vert \\
&\le
\frac{(\sqrt{n}M)^{m+1}}{m!} (\vert \beta \vert + m + 1)!^s 
h^{\vert \beta \vert + m + 1} 
\vert\vert F \vert\vert_{\overline{V} \cap A,s,h} \vert x_1 - x_2 \vert^{m+1}.
\end{aligned}
$$
Moreover, we can suppose $\sqrt{n}M \ge 1$. Hence, if
$|x-y|\leq\delta$, we obtain
$$
\vert\vert F \vert\vert_{K \cap A, s, \sqrt{n}M h}^{\mathcal W}
\le \vert\vert F \vert\vert_{\overline{V} \cap A,s,h} \ .
$$

\

Now, assume that $x_1,x_2\in A \cap K$ satisfy 
$\vert x_1 - x_2 \vert \ge \delta$.
For $h'= (n+1)h$, we have
$$
\begin{aligned}
&\frac{m!\vert R_m(S_\beta F;\,x_1,x_2) \vert}
{h'^{\vert \beta \vert + m + 1}(\vert \beta \vert + m + 1)!^s 
\vert x_1 - x_2\vert^{m+1}} 
\le
\frac{m!(\vert f_\beta(x_1) \vert + \vert T_m(S_\beta F;\, x_1, x_2) \vert) }
{h'^{\vert \beta \vert + m + 1}(\vert \beta \vert + m + 1)!^s 
\vert x_1 - x_2\vert^{m+1}} \\
&
\qquad\le
2\vert\vert F \vert\vert_{\overline{V},s,h}
\sum_{\vert\gamma\vert \le m}
\frac{m! \vert \beta + \gamma\vert!^s h^{\vert \beta + \gamma\vert}
\vert x_1 - x_2\vert^{\vert\gamma\vert}}
{h'^{\vert \beta \vert + m + 1}(\vert \beta \vert + m + 1)!^s \gamma! 
\vert x_1 - x_2 \vert^{m+1}} \\
&
\qquad\le
2\vert\vert F \vert\vert_{\overline{V},s,h} 
\frac{1}{ h'\vert x_1 - x_2\vert} \frac{1}{(n+1)^m}\ \cdot \\
&\qquad\quad\quad
\cdot\sum_{\vert\gamma\vert \le m}
\frac{m!}{(m - \vert\gamma\vert)!\gamma!}
\frac{1}{(m - \vert \gamma \vert)!^{s-1} 
(h\vert x_1 - x_2 \vert)^{m-\vert \gamma \vert}} \\
&\qquad\le
2\vert\vert F \vert\vert_{\overline{V},s,h} 
\frac{1}{h'\delta}\exp(\sigma (h\delta)^{-\sigma}) \frac{1}{(n+1)^m}
\sum_{\vert\gamma\vert \le m}
\frac{m!}{(m - \vert\gamma\vert)!\gamma!} \\
&\qquad
= \frac{2}{h'\delta}\exp(\sigma (h\delta)^{-\sigma}) 
\vert\vert F \vert\vert_{\overline{V},s,h} 
\end{aligned}
$$
where $\sigma = \frac{1}{s-1}$.
The conclusion follows.
\end{proof}

Let $X=\mathbb R^n$, $A\subset X$ a compact subanalytic set and $F
\in \Sswj(X)$.  
Since $A$ is compact, the number of strata of a
stratif\mbox{}ication of $A$ is f\mbox{}inite.  Hence for any $h > 0$
(resp. some $h > 0$) we have $ \vert\vert F \vert\vert_{A,s,h} <
\infty $ if $*=(s)$ (resp. $*=\{s\}$) respectively.  Set
\begin{equation}
  \label{eq:sw_norm_charact}
{\mathcal S}{\mathcal W}^{*}_{A,h}(X) :=
\{F \in {\mathcal S}{\mathcal W}^{*}_{A}(X);
\vert\vert F \vert\vert_{A,s,h} < \infty \},
\end{equation}
and endow ${\mathcal S}{\mathcal W}^{*}_{A,h}(X)$ with the topology
induced by the norm $\vert\vert \cdot \vert\vert_{A,s,h}$.
Then, algebraically, we have 
$${\mathcal S}{\mathcal W}^{(s)}_A(X) \simeq
\underset{h>0}{\varprojlim}\, {\mathcal S}{\mathcal W}^{(s)}_{A,h}(X)
$$
and
$$
{\mathcal S}{\mathcal W}^{\{s\}}_A(X) \simeq
\underset{h>0}{\varinjlim}\, {\mathcal S}{\mathcal W}^{\{s\}}_{A,h}(X).
$$
Therefore, $\Sswj(X)$ can be endowed with a locally convex topology
induced by these projective or inductive limits.

\begin{prop}{\label{prop:whitney-topology}}
Let $X=\mathbb R^n$ and $A$ a 1-regular compact subanalytic subset in $X$.
Then 
$$
\Swj(X) = \Sswj(X) \ .
$$
Moreover, these spaces are topologically isomorphic. In particular, 
$\Sswj(X)$ is an {\bf{FS}} space if $*=(s)$, and
a {\bf{DFS}} space if $*=\{s\}$.
\end{prop}

\begin{proof}
  In Corollary \ref{coro:1-regular-basic}, choose $K$ as $A$ and
  $V$ as an open subanalytic subset containing $A$.
  Then we have
$$
\vert\vert F \vert\vert_{A, s, \kappa h}^{\mathcal W}
\le C\vert\vert F \vert\vert_{A,s,h}.
$$
The result follows.
\end{proof}

\begin{prop}\label{prop:swj_A_=wj_A}
Let $X$ be a real analytic manifold and $A$ a locally closed subanalytic set.
If $A$ is 1-regular at $p \in X$, then 
$
{\Swj}_{,\,p} = {\Sswj}_{,\,p}.
$
\end{prop}

\begin{proof}
Note that for any $p \in X$ we have
$$
{\Swj}_{,\,p} = \lim_{\underset{p \in U}{\longrightarrow}} 
{\mathcal W}^*_{A \cap \overline{U}}(X),\qquad
{\Sswj}_{,\,p} = \lim_{\underset{p \in U}{\longrightarrow}}
{\mathcal S}{\mathcal W}^*_{A \cap \overline{U}}(X).
$$
Further, we have
$$
{\mathcal W}^*_{A \cap \overline{U}}(X) \subset {\mathcal S}{\mathcal W}^*_{A \cap \overline{U}}(X).
$$
On the other hand, for any suf\mbox{}f\mbox{}iciently small open subanalytic
neighborhoods $U_1 \supset\supset U_2 \supset\supset U_3$ of $p$,
Corollary \ref{coro:1-regular-basic} implies that the restriction map
${\mathcal S}{\mathcal W}^*_{A \cap \overline{U}_1}(X) \to {\mathcal
  S}{\mathcal W}^*_{A \cap \overline{U}_3}(X)$ factorizes through ${\mathcal
  W}^*_{A \cap \overline{U}_2}(X)$. Hence the following diagram commutes
$$
\begin{matrix}
{\mathcal S}{\mathcal W}^*_{A \cap \overline{U}_1}(X) & \longrightarrow & 
{\mathcal W}^*_{A \cap \overline{U}_2}(X) \\
 & \searrow &\downarrow \\
 & & {\mathcal S}{\mathcal W}^*_{A \cap \overline{U}_3}(X) \ .
\end{matrix}
$$

The conclusion follows.
\end{proof}

%\begin{defs}
%  Let $A$ be a subanalytic set. We say that a stratif\mbox{}ication
%  $\{A_\alpha\}$ of $A$ is a $1$-regular stratif\mbox{}ication if each
%  stratum $A_\alpha$ is $1$-regular, relatively compact and connected.
%\end{defs}

%Note that any subanalytic subset has a 1-regular stratif\mbox{}ication.
%More precisely, if $\{B_\alpha\}$ is a stratif\mbox{}ication of $A$, then there
%exists a 1-regular stratif\mbox{}ication f\mbox{}iner than $\{B_\alpha\}$.

\begin{coro}
  Let $X$ be a real analytic manifold, $A\subset X$ a locally closed
  subanalytic set and $\{A_\alpha\}$ a 1-regular stratif\mbox{}ication
  of $A$.  Then,  $ \Sswj = \Ssawj.  $
\end{coro}

\begin{proof}
The result comes from the fact that
for any stratum $A_\alpha$, we have %\marginpar{add some explications}
$$
j_{A_\alpha, A}(\Sswj) \subset {\mathcal S}{\mathcal W}^*_{A_\alpha}
=
{\mathcal W}^*_{A_\alpha}.
$$
\end{proof}

\begin{coro}{\label{coro:exact-1-regular}}
  Let $X$ be a real analytic manifold, and let $A_1$ and $A_2$ be
  closed subanalytic subsets, or open subanalytic subsets in $X$. If
  $A_1 \cap A_2$ is 1-regular at $p \in X$, then the sequence
\begin{equation}
  \label{eq:swj_closed_mv}
0 \to {\mathcal S}{\mathcal W}^*_{A_1 \cup A_2,p} \to
{\mathcal S}{\mathcal W}^*_{A_1,p}
\oplus {\mathcal S}{\mathcal W}^*_{A_2,p}
\to
{\mathcal S}{\mathcal W}^*_{A_1 \cap A_2,p}
\to 0
\end{equation}
is exact.
\end{coro}

\begin{proof}
  By Lemma \ref{lemma:stratified-Whitney-exact}, it is
  suf\mbox{}f\mbox{}icient to prove the surjectivity. Since ${\mathcal
    S}{\mathcal W}^*_{A_1 \cap A_2} = {\mathcal W}^*_{A_1 \cap A_2}$
  holds, the result is clear.
\end{proof}

The condition ``$A_1 \cap A_2$ is 1-regular'' in Corollary
\ref{coro:exact-1-regular} is too strong.  Indeed, if $\mathrm{dim}\,
X = 2$, and if $A_1$ and $A_2$ are closed subanalytic subsets, then
\eqref{eq:swj_closed_mv} is always exact, see Theorem
\ref{thm:close_mv_swj}. We also note that, in the case of
$\mathrm{dim} X > 2$, we can f\mbox{}ind an example in which the
surjectivity of the third morphism of \eqref{eq:swj_closed_mv} does not
hold.

\begin{lemma}{\label{lemma:topology-whitney}}
  Let $X$ be a real analytic manifold and $A\subset X$ a compact
  subanalytic set.  Then $\mc {SW}^{(s)}_A(X)$ (resp. $\mc
  {SW}^{\{s\}}_A(X)$) can be endowed with an {\bf{FS}} (resp. a
  {\bf{DFS}}) locally convex topology. Moreover if $A$ is a finite
  union of compact subanalytic sets $B_1$, $\dots$, $B_k$, then the
  canonical morphism
$$
\iota: {\mathcal S}{\mathcal W}^*_{A}(X) 
\to
\underset{1 \le i \le k} \oplus {\mathcal S}{\mathcal W}^*_{B_i}(X)
$$
becomes an injective homomorphism of locally convex topological vector
spaces.
\end{lemma}

\begin{proof}
  Let $\{A_i\}$ be a 1-regular stratif\mbox{}ication of $A$
  satisf\mbox{}ying the following condition. 
  For each stratum $A_i$, there exist an open subset $U \subset X$ containing $\overline{A}_i$
  and isomorphism $\varphi_i: U \to V$ for some open subset $V$ in $\mathbb R^n$.

%Noticing that the number of strata is f\mbox{}inite and 
%${\mathcal S}{\mathcal W}^*_{\overline{A}_i}(X) = {\mathcal W}^*_{\overline{A}_i}(X)$,
It follows from Lemma \ref{lemma:stratified-Whitney-exact}
that the following sequence is exact,
$$
0 \to {\mathcal S}{\mathcal W}^*_{A}(X) 
\to
\underset{1 \le i \le k} \oplus {\mathcal S}{\mathcal W}^*_{\overline{A}_i}(X)
\to
\underset{1 \le i < j \le k} \oplus {\mathcal W}^*_{\overline{A}_i \cap \overline{A}_j}(X)
\subset
\underset{1 \le i < j \le k} \oplus {\mathcal S}{\mathcal
  W}^*_{\overline{A}_i \cap \overline{A}_j}(X) \ .
$$
We can consider the sets $\overline{A}_i$ as 1-regular compact subanalytic
subsets of $\mathbb R^n$. Hence ${\mathcal S}{\mathcal
  W}^*_{\overline{A}_i}(X)$ has an {\bf{FS}} or a {\bf{DFS}} locally convex
topology by Proposition \ref{prop:whitney-topology}, and the morphism
$$
\underset{1 \le i \le k} \oplus {\mathcal S}{\mathcal W}^*_{\overline{A}_i}(X) =
\underset{1 \le i \le k} \oplus {\mathcal W}^*_{\overline{A}_i}(X) 
\to
\underset{1 \le i < j \le k} \oplus {\mathcal W}^*_{\overline{A}_i \cap \overline{A}_j}(X)
$$
is continuous for such topologies.  We endow $\Sswj(X)$ with the
induced topology.  By the exactness of the above sequence, the
topological space ${\mathcal S}{\mathcal W}^*_{A}(X)$ is a closed
subspace of $ \underset{1 \le i \le k} \oplus {\mathcal S}{\mathcal
  W}^*_{\overline{A}_i}(X).  $ Therefore ${\mathcal S}{\mathcal
  W}^*_{A}(X)$ is an {\bf{FS}} or a {\bf{DFS}} space.

%Taking Lemma \ref{lemma:equiv-norm} into account,
One can check that another choice of 1-regular stratifications and
morphisms $\varphi_i$ induces an equivalent topology. Indeed, by
considering a 1-regular stratification finer than those, we can reduce the problem
to the following claim: Let $A \subset \mathbb R^n$ be a 
compact 1-regular subanalytic subset and $A_i \subset A (\subset \mathbb R^n)$ ($i=1,2,\dots,k$) compact
1-regular subanalytic subsets with $A = \cup A_i$. Then the canonical morphism
$$
\iota: {\mathcal S}{\mathcal W}^*_{A}(X) 
\to
\underset{1 \le i \le k} \oplus {\mathcal S}{\mathcal W}^*_{A_i}(X)
$$
is a homomorphism of locally convex topological vector spaces.

If $*=(s)$, then these vector spaces have {\bf{FS}} topologies and the
image of $\iota$ is closed by Lemma
\ref{lemma:stratified-Whitney-exact}. Hence the claim follows from the
open mapping theorem. 

Now, let us prove the claim for $*=\{s\}$. Since a {\bf{DFS}} space is
bornological, it suffices to show that for a sequence $\{x_j\}_{j \in
  \mathbb N} \subset {\mathcal S}{\mathcal W}^{\{s\}}_{A}(X)$ with
$\iota(x_j) \to 0$ ($j \to \infty$), the sequence $\{x_j\}$ also tends
to $0$.  Since $\{\iota(x_j)\}$ is bounded in the {\bf{DFS}} space
$\underset{1 \le i \le k} \oplus {\mathcal S}{\mathcal
  W}^{\{s\}}_{A_i}(X)$, there exists an $h > 0$ with $ \iota(x_j) \in
\underset{1 \le i \le k} \oplus {\mathcal S}{\mathcal
  W}^{\{s\}}_{A_i,h}(X) $ for any $j$.  Then, as $A = \cup A_i$, we
have the estimate
$$
\vert\vert x \vert\vert_{A,s,h}
\le \sum_{i=0}^k \vert\vert \iota(x) \vert\vert_{A_i,s,h},
$$
and from which the claim follows. \par
The last assertion in the lemma can be proved in the similar way.
\end{proof}
If $X=\mathbb R^n$, $A\subset X$ a compact subanalytic set and
$*=(s)$, the {\bf{FS}} topology in ${\mathcal S}{\mathcal
  W}^{(s)}_A(X)$ is described as follows.  Given a sequence
$\{F_n\}\subset{\mathcal S}{\mathcal W}^{(s)}_A(X)$, one has that
$$
\ou{n\to\infty}{\mathrm{lim}}F_n \to 0 \iff \vert\vert F_n
\vert\vert_{A,s,h} \to 0 \text{ for any } h > 0.
$$
Note that the convergence in ${\mathcal W}^{(s)}_A(X)$ is
def\mbox{}ined in the following way.  Given a sequence
$\{G_n\}\subset{\mathcal W}^{(s)}_A(X)$,
$$
\ou{n\to\infty}{\mathrm{lim}}G_n \to 0 \iff 
\vert\vert G_n \vert\vert_{A,s,h} \to 0 \text { and }
\vert\vert G_n \vert\vert^{\mathcal W}_{A,s,h} \to 0
\text{ for any } h > 0.
$$
These two topologies coincides if $A$ is 1-regular. 

\subsection{The subanalytic sheaf of the stratif\mbox{}ied Whitney jets}
Let $X$ be a real analytic manifold.  The subanalytic presheaf of
stratif\mbox{}ied Whitney jets of class $*$ is def\mbox{}ined by
$$
{\mathcal S}{\mathcal W}^*_{X_{sa}}(U) := 
{\mathcal S}{\mathcal W}^*_U(X) \ ,
$$
where $U$ is a subanalytic open subset of $X$. 

It follows from Lemma \ref{lemma:stratified-Whitney-exact} that
${\mathcal S}{\mathcal W}^*_{X_{sa}}$ is a subanalytic sheaf in
$X_{sa}$.  Since $\Gamma(\overline{U}, {\mathcal D}_X)$ acts on ${\mathcal
  S}{\mathcal W}^*_U(X)$, ${\mathcal S}{\mathcal W}^*_{X_{sa}}$ is a
$\rho_!{\mathcal D}_X$-module.

\begin{prop}\label{prop:wj=swj_1-reg}
If $U\subset X$ is a 1-regular open subanalytic set, then 
$$
{\mathcal S}{\mathcal W}^*_{X_{sa}}(U) 
= {\mathcal W}^*_U(X) \simeq
{\mathcal W}^*_{\overline{U}}(X) \simeq
\frac{ {\ms C}^*(X) }{ {\mathcal I}^*_{X,{\overline{U}}} (X)} \ ,
$$
where ${\mathcal I}^*_{X,\overline{U}}$ denotes the subsheaf of $\ms C^*$
consisting of functions vanishing on $\overline{U}$ up to inf\mbox{}inite
order.
\end{prop}

\begin{coro}
  For $U\in\Op^c(\xsa)$, set $\mc W^*_{\xsa}(U):=\mc W^*_U(X)$. Then
  $\mc W^{*a}_{\xsa}\simeq\mc {SW}^*_\xsa$.
\end{coro}

\begin{proof}
  It is sufficient to combine Proposition \ref{prop:wj=swj_1-reg},
  Lemma \ref{lemma:sheafification} and Proposition \ref{prop:kurdyka}.
\end{proof}

\subsection{Stratif\mbox{}ied ultradistributions}
Let $X$ be a real analytic manifold, $A$ a closed subanalytic subset
in $X$ and let ${\mathcal D}b^*$ denote the sheaf of
ultradistributions of class $*$.  For any stratif\mbox{}ication
$\{A_\alpha\}$ of $A$, let us def\mbox{}ine stratif\mbox{}ied
ultradistributions along $\{A_\alpha\}$.
\begin{defs}
An ultradistribution $u \in {\mathcal D}b^*(U)$ is said to be
stratif\mbox{}ied along $\{A_\alpha\}$ in $U$ if $u$ can be written 
in the form:
$$
u=\sum_{\alpha} u_{\alpha},\qquad 
u_\alpha \in \Gamma_{\overline{A}_\alpha}(U, {\mathcal D}b^*).
$$
\end{defs}
We def\mbox{}ine the sheaf of \emph{stratif\mbox{}ied
ultradistributions of class $*$ along $\{A_\alpha\}$} as 
$$
{\mathcal S}{\mathcal D}b^*_{\{A_\alpha\}}(U)
:= \{u \in {\mathcal D}b^*(U);\, \text{$u$ is stratif\mbox{}ied along $\{A_\alpha\}$ 
in $U$}\} \ .
$$
For a stratif\mbox{}ication $\{A'_\alpha\}$ f\mbox{}iner than
$\{A_\alpha\}$, there exists the canonical morphism
$$
{\mathcal S}{\mathcal D}b^*_{\{A'_\alpha\}}(U)
\to
{\mathcal S}{\mathcal D}b^*_{\{A_\alpha\}}(U) \ .
$$
We def\mbox{}ine the sheaf of \emph{stratif\mbox{}ied
ultradistributions of class $*$ along $A$} as
$$
{\mathcal S}{\mathcal D}b^*_{[A]}(U) :=
\underset{\begin{subarray}{c}\text{stratif\mbox{}ication}\\ \{A_\alpha\}\text{ of }
      A\end{subarray}}{\varprojlim} {\mathcal S}{\mathcal
  D}b^*_{\{A_\alpha\}}(U) \ .
$$
Since for any stratif\mbox{}ication $\{A_\alpha\}$ there exists a
1-regular stratif\mbox{}ication f\mbox{}iner than $\{A_\alpha\}$, we
have
$$
{\mathcal S}{\mathcal D}b^*_{[A]}(U) =
\underset{\begin{subarray}{c}\text{1-regular stratif\mbox{}ication}\\
      \{A_\alpha\} \text{ of } A\end{subarray}}{\varprojlim} {\mathcal
  S}{\mathcal D}b^*_{\{A_\alpha\}}(U) \ .
$$

There exists the canonical injective sheaf homomorphism
$$
{\mathcal S}{\mathcal D}b^*_{[A]} \hookrightarrow \Gamma_A({\mathcal D}b^*) \ .
$$
This morphism is not surjective in general.
The following lemma follows easily from the def\mbox{}inition.
\begin{lemma}{\label{lemma:stratified-ultradistributions-exact}}
Let $X$ be a real analytic manifold, and let $A_1$, $\dots$, $A_l$ be
closed subanalytic subsets in $X$.
Then the sheaf homomorphism
$$
\oplus {\mathcal S}{\mathcal D}b^*_{[A_i]} \to
{\mathcal S}{\mathcal D}b^*_{[\cup A_i]}
$$
is surjective.
\end{lemma}

%\begin{rem}
Remark that, in general, the middle of the sequence
$$
0
\to
{\mathcal S}{\mathcal D}b^*_{[A_1\cap A_2]}
\to
{\mathcal S}{\mathcal D}b^*_{[A_1]}
\oplus 
{\mathcal S}{\mathcal D}b^*_{[A_2]}
\to
{\mathcal S}{\mathcal D}b^*_{[A_1 \cup A_2]}
\to 0
$$
is not exact.
%\end{rem}

\begin{theorem}\label{thm:dual_swj}
  Let $X$ be a real analytic manifold, and $A\subset X$ a compact
  subanalytic set. Then, algebraically, we have 
$$
{\mathcal S}{\mathcal D}b^*_{[A]}(X) \simeq (\Sswj\Volo(X))'
$$
where $(\Sswj\Volo(X))'$ denotes the topological dual space of $\Sswj\Volo(X)$
and ${\mathcal V}_X$ designates the sheaf of volume elements in $X$, i.e.,
$\omega^{(n)}_X \underset{\mathbb Z}{\otimes} {or}_X$.
\end{theorem}

\begin{proof}
The continuous morphism $j_A: {\ms C}^*(X) \to \Sswj(X)$ induces
the morphism
$$
j^t_A: (\Sswj\Volo(X))' \to 
({\ms C}^*\Volo(X))' \subset {\mathcal D}b^*(X).
$$
For any $\varphi(x) \in {\ms C}^*(X)$ with
$\operatorname{supp}(\varphi) \cap A = \varnothing$, we have
$j_A(\varphi(x)) = 0$. Hence $\operatorname{im} j^t_A \subset
\Gamma_A(X, {\mathcal D}b^*)$.  Moreover since $j_A({\ms C}^*(X))$ is
dense in $\Sswj(X)$ by Proposition {\ref{prop:dense-whitney}}, the morphism
$$
j^t_A: (\Sswj\Volo(X))' \hookrightarrow \Gamma_A(X, {\mathcal D}b^*).
$$
is injective.

Let $A_1$, $A_2$, $\dots$, $A_l$ be closed subanalytic subsets in $X$
with $\cup A_i = A$.
If we prove that, for each $i$, 
$$
j^t_{A_i}( ({\mathcal S}{\mathcal W}^*_{A_i}\Volo(X))') =
{\mathcal S}{\mathcal D}b^*_{[A_i]}(X) \ ,
$$
 then 
$$
j^t_{A}( ({\mathcal S}{\mathcal W}^*_{A}\Volo(X))') =
{\mathcal S}{\mathcal D}b^*_{[A]}(X)
$$
follows from the following commutative diagram
\begin{equation}
  \label{eq:diagram}
\begin{matrix}
  0        &  & 0 & & \\
\downarrow &  & \downarrow & & \\
\oplus ({\mathcal S}{\mathcal W}^*_{A_i}\Volo(X))' & \longrightarrow & 
(\Sswj\Volo(X))'  & 
\longrightarrow & 0 \\
\downarrow     &  & \downarrow & & \\
\oplus\Gamma_{A_i}(X,{\mathcal D}b^*) & \longrightarrow & 
\Gamma_A(X, {\mathcal D}b^*)  & & \\
\uparrow &  & \uparrow  & & & \\
\oplus {\mathcal S}{\mathcal D}b^*_{[A_i]}(X) & \longrightarrow & 
{\mathcal S}{\mathcal D}b^*_{[A]}(X) & \longrightarrow & 0 & \\
\uparrow           &  & \uparrow & & \\
    0       &  & 0 & &
\end{matrix}
\end{equation}
The first row of \eqref{eq:diagram} is exact since
$$
%0 \to 
\Sswj\Volo(X) \to \oplus {\mathcal S}{\mathcal W}^*_{A_i}\Volo(X)
$$
is an injective homomorphism of locally topological vector spaces by
Lemma \ref{lemma:topology-whitney}.
The third row of \eqref{eq:diagram} is exact by Lemma
\ref{lemma:stratified-ultradistributions-exact}.
All vertical arrows of \eqref{eq:diagram} are injective.

By these observations, we can reduce the problem to the case
$X=\mathbb R^n$ and $A\subset X$ is a compact subanalytic set.
F\mbox{}irst recall that if $B$ is a compact set in $\mathbb R^n$,
then it follows from the result of Whitney and Kantor that
$$
j_B^t: ({\mathcal W}^*_B(X))' \simeq \Gamma_B(X, {\mathcal D}b^*) \ .
$$
Let $\{A_\alpha\}$ be a $1$-regular stratif\mbox{}ication of $A$. Let
us consider the following commutative diagram
$$
\begin{matrix}
  & &   &  & 0 \\
  & &   &  & \downarrow \\
0 & \longrightarrow & {\mathcal W}_A^*(X) & \longrightarrow & 
{\mathcal S}{\mathcal W}_A^*(X) \\
& & \downarrow & & \downarrow \\
& & \oplus {\mathcal W}_{ \overline{A}_\alpha}^*(X) 
& \simeq
&
\oplus {\mathcal S}{\mathcal W}_{ \overline{A}_\alpha}^*(X) \ .
\end{matrix}
$$
Here the f\mbox{}irst horizontal arrow is injective and has a dense image
by Proposition {\ref{prop:dense-whitney}}.
Since each $\overline{A}_\alpha$ is 1-regular,
the second horizontal arrow is topologically isomorphic.
The second vertical arrow is an injective homomorphism of
locally convex topological vector spaces.
Then taking the dual of the diagram, we have
$$
\begin{matrix}
 & & 0 \\
 & & \uparrow \\
\Gamma_A(X, {\mathcal D}b^*) & \overset{j^t_A}{\longleftarrow} &
(\Sswj(X))' & \longleftarrow & 0 \\
\uparrow & & \uparrow  \\
\oplus \Gamma_{\overline{A}_\alpha}(X, {\mathcal D}b^*) & 
\overset{j^t_{\overline{A}_\alpha}}{\simeq} &
\oplus ({\mathcal S}{\mathcal W}^*_{\overline{A}_\alpha}(X))'.
\end{matrix}
$$
Hence we can conclude 
\begin{equation}{\label{eq:dual-1-regular-1}}
j^t_A( (\Sswj(X))') = \operatorname{Im}
\left(\oplus \Gamma_{\overline{A}_\alpha}(X, {\mathcal D}b^*)\right)
= {\mathcal S}{\mathcal D}b^*_{\{A_\alpha\}}(X)
\end{equation}
for any 1-regular stratif\mbox{}ication $\{A_\alpha\}$ of $A$.
In particular, we obtain
\begin{equation}{\label{eq:dual-1-regular-2}}
j^t_A( (\Sswj(X))') = \varprojlim\,
{\mathcal S}{\mathcal D}b^*_{\{A_\alpha\}}(X)
= {\mathcal S}{\mathcal D}b^*_{[A]}(X).
\end{equation}
The conclusion follows.

%Then the above diagram becomes:
%$$
%\begin{matrix}
% 0 & & 0 \\
%\uparrow & & \uparrow \\
%{\mathcal S}{\mathcal D}b^*_{[A]}(X) & \overset{j^t_A}{\longleftarrow} &
%(\Sswj(X))' & \longleftarrow & 0 \\
%\uparrow & & \uparrow  \\
%\oplus {\mathcal S}{\mathcal D}b^*_{[\overline{A}_\alpha]}(X) & 
%\overset{j^t_{\overline{A}_\alpha}}{\simeq} &
%\oplus ({\mathcal S}{\mathcal W}^*_{\overline{A}_\alpha}(X))'
%\end{matrix}
%$$
%where the second horizontal arrows is still isomorphic because
%the composition of morphisms
%$$
%({\mathcal S}{\mathcal W}^*_{\overline{A}_\alpha}(X))' \hookrightarrow
%{\mathcal S}{\mathcal D}b^*_{[\overline{A}_\alpha]}(X) \hookrightarrow
%\Gamma_{\overline{A}_\alpha}(X, {\mathcal D}b^*) 
%$$
%is an isomorphism. The f\mbox{}irst columns is exact due to Lemma XXX.
%
%Then by the above commutative diagram, we can conclude that 
%$$
%j^t_A: (\Sswj(X))' \hookrightarrow {\mathcal S}{\mathcal D}b^*_{[A]}(X)
%$$
%is surjective.
\end{proof}

\begin{coro}{\label{coro:1-regular-equiv}}
  Let $X$ be a real analytic manifold, $A\subset X$ a closed
  subanalytic set. If $\{A_\alpha\}$ is a 1-regular
  stratif\mbox{}ication of $A$, then we have
$$
{\mathcal S}{\mathcal D}b^*_{[A]} =
{\mathcal S}{\mathcal D}b^*_{\{A_\alpha\}}.
$$
In particular, if $A$ is 1-regular at $p \in X$, then we have
$$
\left({{\mathcal S}{\mathcal D}b^*_{[A]}}\right)_{p} = 
\left({\Gamma_A({\mathcal D}b^*)}\right)_{p}.
$$
\end{coro}

\begin{proof}
Set
  $
Z(A_\beta) := \underset{\overline{A}_\gamma \cap \overline{A}_\beta \ne \varnothing}
\cup \overline{A}_\gamma$.
  For any $p \in A_\beta$, since $Z(A_\beta)$ is a closed neighborhood
  of $p$ in $A$, we have
$$
\left({{\mathcal S}{\mathcal D}b^*_{\{A_\alpha\}}}\right)_p 
=
\left({{\mathcal S}{\mathcal D}b^*_{\{A_\alpha \cap Z(A_\beta)\}}}\right)_p, \qquad
\left({{\mathcal S}{\mathcal D}b^*_{[A]}}\right)_p
=
\left({{\mathcal S}{\mathcal D}b^*_{[A \cap Z(A_\beta)]}}\right)_p.
$$
Hence we may assume that $A$ is compact.
Since
$
{\mathcal S}{\mathcal D}b^*_{[A]} 
\subset
{\mathcal S}{\mathcal D}b^*_{\{A_\alpha\}},
$
it is enough to show that
$$
\left({{\mathcal S}{\mathcal D}b^*_{[A]}}\right)_p 
\hookrightarrow
\left({{\mathcal S}{\mathcal D}b^*_{\{A_\alpha\}}}\right)_p 
$$
is surjective.
By the softness of the sheaf ${\mathcal S}{\mathcal D}b^*_{\{A_\alpha\}}$,
$$
{\mathcal S}{\mathcal D}b^*_{\{A_\alpha\}}(X) \to 
\left({\mathcal S}{\mathcal D}b^*_{\{A_\alpha\}}\right)_p
$$
is surjective. Hence it is suf\mbox{}f\mbox{}icient to show $
{{\mathcal S}{\mathcal D}b^*_{\{A_\alpha\}}}(X) = {{\mathcal
    S}{\mathcal D}b^*_{[A]}}(X).  $ This follows from the equations
\eqref{eq:dual-1-regular-1} and \eqref{eq:dual-1-regular-2}.
\end{proof}

\begin{coro}\label{coro:decomposability_sdb}
  Let $X$ be a real analytic manifold, $A_1,A_2\subset X$ closed
  subanalytic sets. If $A_1 \cap A_2$ is 1-regular at $p \in X$, then
  the sequence
\begin{equation}
  \label{eq:sdb_closed_mv}
0 \to {\mathcal S}{\mathcal D}b^*_{[A_1 \cap A_2],p} \to
{\mathcal S}{\mathcal D}b^*_{[A_1],p}
\oplus {\mathcal S}{\mathcal D}b^*_{[A_2],p}
\to
{\mathcal S}{\mathcal D}b^*_{[A_1 \cup A_2],p}
\to 0
\end{equation}
is exact.
\end{coro}

\begin{proof}
The injectivity is clear, and the surjectivity comes from
Lemma \ref{lemma:stratified-ultradistributions-exact}.
The exactness of the middle follows from 
$$
{\mathcal S}{\mathcal D}b^*_{[A_1]}
\cap {\mathcal S}{\mathcal D}b^*_{[A_2]}
\subset
\Gamma_{A_1 \cap A_2}({\mathcal D}b^*)
$$
and
$
\Gamma_{A_1 \cap A_2}({\mathcal D}b^*)
={\mathcal S}{\mathcal D}b^*_{[A_1 \cap A_2]}
$
at $p$.
\end{proof}

%\begin{rem}
Similarly to the exactness of \eqref{eq:swj_closed_mv}, the exactness of
\eqref{eq:sdb_closed_mv} holds if $\mathrm{dim}\,X \le 2$ without
the assumption of $1$-regularity on $A_1\cap A_2$.

%, For the stratif\mbox{}ied ultradistributions, the same remark as in the case of 
%the stratif\mbox{}ied Whitney jets holds, that is,
%if $\operatorname{dim} X = 2$, then the above sequence is always exact.
%\end{rem}

If $A$ is a compact subanalytic, then ${\mathcal S}{\mathcal D}b^*_{[A]}(X)$ is
equipped with the strong dual topology of the locally convex topological vector space 
$\Sswj(X)$.
Then ${\mathcal S}{\mathcal D}b^*_{[A]}(X)$ is a {\bf{DFS}} (resp. an {\bf{FS}}) space
if $*=(s)$ (resp. $*=\{s\}$) respectively. 
Since ${\mathcal S}{\mathcal D}b^*_{[A]}(X)$ and $\Gamma_A(X, {\mathcal D}b^*)$
are reflexive,
$
{\mathcal S}{\mathcal D}b^*_{[A]}(X)
$
is dense in $\Gamma_A(X, {\mathcal D}b^*)$.

\section{Subanalytic sheaves on real surfaces}\label{section:2dim}

In this section we are going to study in detail the extension
properties of stratif\mbox{}ied Whitney jets on real surfaces. 

\subsection{On the exactnesses of \eqref{eq:swj_closed_mv} and
  \eqref{eq:sdb_closed_mv} in dimension 2}

%Let $X$ be real analytic manifold with $\operatorname{dim}_{\mathbb R} X = 2$
%and 
Throughtout the subsection $X$ is a real analytic manifold of dimension $2$, unless
otherwise specified.

\begin{defs}
  Let $A\subset X$ be a closed subanalytic set. We say that a
  1-regular stratif\mbox{}ication $\{A_\alpha\}$ of $A$ is \emph{good} if
  every $\overline{A}_\alpha$ is topologically isomorphic to
  $D^{\operatorname{dim}_{\mathbb R} A_\alpha}$ as a topological
  manifold with the boundary, where $D^{k}$ denotes a closed unit disc
  in $\mathbb R^k$.
\end{defs}

By \cite{kurdyka}, for any stratif\mbox{}ication $\{A_\alpha\}$ of
$A$, there exists a good 1-regular stratif\mbox{}ication f\mbox{}iner
than $\{A_\alpha\}$.

\begin{lemma}{\label{lemma:outer-extension}}
  Let $A\subset X$ be a closed subanalytic set, $\{A_\alpha\}$ a good
  1-regular stratif\mbox{}ication of $A$.  For any $A_\alpha$ with
  $\operatorname{dim} A_\alpha = 2$, the restriction map
$$
{\mathcal S}{\mathcal W}^*_{X\setminus A_\alpha}(X) \to 
{\mathcal S}{\mathcal W}^*_{\overline{A}_\alpha\setminus A_\alpha}(X)
$$
is surjective.
\end{lemma}

\begin{proof} 
  Let $A_\alpha$ satisfy $\dim A_\alpha=2$.  Set $Z:=\overline{A}_\alpha
  \setminus A_\alpha$ and let $\{Z_\beta\}$ be the induced good
  1-regular stratif\mbox{}ication of $Z$.  For any $p$ and $\epsilon > 0$,
  $D_\epsilon(p)$ designates the closed disk with center $p$ and
  radius $\epsilon$. By the partition of unity, it is enough to show
  that for any $p \in Z$, there exists $\ep>0$ such that
$$
{\mathcal S}{\mathcal W}^*_{D_\epsilon(p) \setminus A_\alpha}(X) \to 
{\mathcal S}{\mathcal W}^*_{Z \cap D_\epsilon(p)}(X)
$$
is surjective. 

If $p \in Z_\beta$ with $\operatorname{dim}Z_\beta = 1$, then $Z \cap
D_\epsilon(p)$ is 1-regular for suf\mbox{}f\mbox{}iciently small
$\epsilon > 0$.  The result is clear in this case.

Suppose, now, $Z_\beta = \{p\}$. Since $Z$ is topologically trivial,
there exist only two strata $Z_1$ and $Z_2$ such that
$\operatorname{dim}{Z_i} = 1$ and $p \in \overline{Z}_i$ ($i=1,2,$). Let
$\epsilon > 0$ be such that $Z_1$ and $Z_2$ cross $\partial
D_\epsilon(p)$ transversally and any stratum other than $Z_1,Z_2$ and
$Z_\beta$ does not intersect with $D_\epsilon(p)$.

Since $\overline{A}_\alpha$ is 1-regular, the angle between the tangent
lines of $Z_1$ and $Z_2$ at $p$ in the side of $D_\epsilon(p)
\setminus A_\alpha$ is positive. Hence, if $\epsilon$ is
suf\mbox{}f\mbox{}iciently small, then there exists $q \in
\partial D_\epsilon(p) \setminus A_\alpha$ such that the segment $l$
from $p$ to $q$ is contained in $D_\epsilon(p) \setminus A_\alpha$,
and that $Z_i$ is not tangent to $l$ ($i=1,2$). One checks easily that
$\mathring{D_\ep(p)}\setminus(l\cup\overline Z_1\cup\overline Z_2)$ has three
connected components, one of whose is $\mathring A_\alpha\cap
B_\ep(p)$. Denote by $W_1$ and $W_2$ the other two connected
components. 
% Then the closed set
%$D_\epsilon(p) \setminus A_\alpha$ is divided into two closed sets
The sets $W_1$ and $W_2$ satisfy
\begin{enumerate}
\item $\overline W_1 \cap\overline W_2 = l$ and $\overline W_1 \cup\overline  W_2 = D_\epsilon(p) \setminus A_\alpha$,
\item the boundary $\partial W_i$ of $W_i$ consists of $Z_i$, $l$ 
and a part of the circle, in particular, 
$W_i$ and $\partial W_i$ are 1-regular ($i=1,2$).
\end{enumerate}

Let $F \in {\mathcal S}{\mathcal W}^*_{Z \cap D_\epsilon(p)}(X)$.  For
sake of simplicity, we assume that $F\vert_{Z_\beta} = 0$ and
$F\vert_{Z \cap \partial D_\epsilon(p)} = 0$. Then,
we define $F_i \in {\mathcal S}{\mathcal W}^*_{\partial W_i}(X)$
($i=1,2$) by
$$
F_i(x) := \left\{
\begin{matrix}
F(x) \qquad& \text{if }x \in Z_i \ , \\
0	     \qquad& \text{if }x \notin Z_i \ .
\end{matrix}
\right.
$$
Since $\partial W_i$ is 1-regular, we can f\mbox{}ind a function 
$\varphi_i(x) \in {\ms C}^*(X)$
such that $j_{\partial W_i}(\varphi_i) = F_i$.
Noticing 
$$
j_{\overline W_1} \varphi_1 \vert_{\overline W_1 \cap\overline  W_2} = j_{\overline W_2} \varphi_2 \vert_{\overline W_1 \cap\overline  W_2} =0 \ ,
$$
the jet
$$
G(x) := 
\left\{
\begin{matrix}
j_{\overline W_1}(\varphi_1) \qquad& x \in \overline W_1,\\
j_{\overline W_2}(\varphi_2) \qquad& x \in \overline W_2
\end{matrix}
\right.
$$
belongs to ${\mathcal S}{\mathcal W}^*_{D_\epsilon(p) \setminus A_\alpha}(X)$, and
$G(x) \vert_{Z \cap D_\epsilon(p)} = F$.
\end{proof}

By the similar arguments as in the proof of Lemma \ref{lemma:outer-extension},
we can also prove the following lemma.
\begin{lemma}{\label{lemma:dual-outer-extension}}
  Let $A\subset X$ be a closed subanalytic set, $\{A_\alpha\}$ a good
  1-regular stratif\mbox{}ication of $A$.  For any $A_\alpha$ with
  $\operatorname{dim} A_\alpha = 2$, we have
$$
{\mathcal S}{\mathcal D}b^*_{[\overline{A}_\alpha \setminus A_\alpha]}(X) =
{\mathcal S}{\mathcal D}b^*_{[X \setminus A_\alpha]}(X) \cap
\Gamma_{\overline{A}_\alpha \setminus A_\alpha}(X,\, {\mathcal D}b^*).
$$
\end{lemma}

\begin{theorem}\label{thm:close_mv_swj}
Let $Z_1,Z_2\subset X$ be closed subanalytic sets. The
  sequence
  \begin{equation}
    \label{eq:swj_closed_mv_surj}
0 \to {\mathcal S}{\mathcal W}^*_{Z_1 \cup Z_2}
\to 
{\mathcal S}{\mathcal W}^*_{Z_1} \oplus
{\mathcal S}{\mathcal W}^*_{Z_2} 
\to
{\mathcal S}{\mathcal W}^*_{Z_1 \cap Z_2} \to 0
\end{equation}
is exact.
\end{theorem}
\begin{proof}
  Since it is a local problem, we may assume that $X=\mathbb R^2$ and
  $Z_i$ is compact.  Set $Z:= Z_1 \cup Z_2$. Let $\{Z_\alpha\}_{\alpha
    \in \Lambda}$ be a good 1-regular stratif\mbox{}ication of $Z_1
  \cup Z_2$ f\mbox{}iner than the partition $\{Z_1 \cup Z_2,\, Z_1,\,
  Z_2,\, Z_1 \cap Z_2\}$. Note that $\Lambda$ is a f\mbox{}inite set.

  We will prove the assertion by induction of the cardinality of
  $\Lambda$.

  By Lemma \ref{lemma:stratified-ultradistributions-exact}, it is
  enough to show the exactness of the sequence
$$
{\mathcal S}{\mathcal W}^*_{Z_1}(X) \oplus
{\mathcal S}{\mathcal W}^*_{Z_2}(X) 
\to
{\mathcal S}{\mathcal W}^*_{Z_1 \cap Z_2}(X) \to 0.
$$

Let $\beta \in \Lambda$ be such that 
$$
\operatorname{dim} (Z_1 \cup Z_2) = \operatorname{dim} Z_\beta \ . 
$$
For $i=1,2$, set 
$$
Z_i' := Z_i \setminus Z_\beta \ .
$$
Note that $Z_i'$ is a closed subanalytic set and
\begin{equation}{\label{eq:boundary-xxx}}
Z_\beta \subset Z_i \Longrightarrow
Z_i' \cap \overline{Z}_\beta = \partial Z_\beta \qquad (i = 1,2)\ .
\end{equation}

The sequence
$$
{\mathcal S}{\mathcal W}^*_{Z'_1}(X) \oplus
{\mathcal S}{\mathcal W}^*_{Z'_2}(X) 
\to
{\mathcal S}{\mathcal W}^*_{Z'_1 \cap Z'_2}(X) 
={\mathcal S}{\mathcal W}^*_{(Z_1 \cap Z_2) \setminus Z_\beta}(X) \to 0
$$
is exact by the induction hypothesis.

Let $F \in {\mathcal S}{\mathcal W}^*_{Z_1 \cap Z_2}(X)$.
For sake of simplicity, we assume 
$$
F\vert_{\overline{Z}_{\beta}} = 0 \ .
$$
It follows from the above exact sequence that there exist $ F_1 \in
{\mathcal S}{\mathcal W}^*_{Z'_1}(X),\ F_2 \in {\mathcal S}{\mathcal
  W}^*_{Z'_2}(X)$ such that
$$
 F_1\vert_{(Z_1 \cap Z_2)\setminus Z_{\beta}} - F_2\vert_{(Z_1 \cap Z_2)\setminus Z_{\beta}}=F\vert_{(Z_1 \cap Z_2) \setminus Z_{\beta}} \ .
$$
Suppose $Z_\beta \subset Z_1$.  Then if $\operatorname{dim}(Z_\beta) =
2$, by Lemma \ref{lemma:outer-extension} there exists $\tilde{F} \in
\mathcal S\mathcal W^*_{(Z_1 \cup Z_2)\setminus Z_\beta}(X)$ such that
$$
\tilde{F}\vert_{\partial Z_\beta} = F_1 \vert_{\partial Z_\beta} \ .
$$
Moreover, if $\dim Z_\beta<2$, then $\partial Z_\beta$ consists of
isolated points, hence there exists $\tilde{F} \in \mathcal S\mathcal
W^*_{(Z_1 \cup Z_2)\setminus Z_\beta}(X)$ such that
$\tilde{F}\vert_{\partial Z_\beta} = F_1 \vert_{\partial Z_\beta}$.

Set
$$
\tilde{F}_1 := F_1 - \tilde{F}\vert_{Z_1'},\qquad
\tilde{F}_2 := F_2 - \tilde{F}\vert_{Z_2'} \ .
$$
Remark that 
$$
\tilde{F}_1\vert_{(Z_1 \cap Z_2)\setminus Z_{\beta}} - \tilde{F}_2\vert_{(Z_1 \cap Z_2)\setminus Z_{\beta}} = F\vert_{(Z_1 \cap Z_2)\setminus Z_{\beta}} \ .
$$ 

Taking (\ref{eq:boundary-xxx}) and $\tilde{F}_1\vert_{\partial
  Z_\beta} = 0$ into account, we can extend $\tilde{F}_1$ to an
element of ${\mathcal S}{\mathcal W}^*_{Z_1}(X)$ by zero (i.e. the
zero extension). Now, if $Z_\beta \cap Z_2 = \varnothing$, then
$Z_2=Z_2'$ and the result follows. Otherwise, suppose $Z_\beta \subset
Z_2$, then, since $\tilde{F}\vert_{\partial Z_\beta} =
\tilde{F_1}\vert_{\partial Z_\beta} = 0$, we have
$\tilde{F}_2\vert_{\partial Z_\beta} = 0$.  Hence $\tilde{F}_2$ is
regarded as an element in ${\mathcal S}{\mathcal W}^*_{Z_2}(X)$ by the
zero extension.

The conclusion follows.
\end{proof}

%\begin{rem}
In general, for open subanalytic subsets $U_1$ and $U_2$, the sequence
$$
0 \to {\mathcal S}{\mathcal W}^*_{U_1 \cup U_2}
\to 
{\mathcal S}{\mathcal W}^*_{U_1} \oplus
{\mathcal S}{\mathcal W}^*_{U_2} 
\to
{\mathcal S}{\mathcal W}^*_{U_1 \cap U_2} \to 0
$$
is not exact, indeed the surjectivity does not hold.  The lack of
surjectivity is of topological nature and it comes from the fact that
$R{\mathcal H}om({\mathbb C}_{U_1 \cup U_2},\, {\mathbb C}_X)$ is not
necessarily concentrated in degree $0$.
%\end{rem}

\begin{coro}{\label{coro:ultra-exact}}
  Let $Z_1,Z_2\subset X$ be closed subanalytic sets. The
  sequence
$$
0 \to {\mathcal S}{\mathcal D}b^*_{[Z_1 \cap Z_2]}
\to 
{\mathcal S}{\mathcal D}b^*_{[Z_1]} \oplus
{\mathcal S}{\mathcal D}b^*_{[Z_2]} 
\to
{\mathcal S}{\mathcal D}b^*_{[Z_1 \cup Z_2]} \to 0
$$
is exact.
\end{coro}

\begin{proof}
  Since it is a local problem, we may assume that $X=\mathbb R^2$ and
  $Z_i$ is compact subanalytic. It suffices to show the exactness of
  the sequence:
$$
0 \to {\mathcal S}{\mathcal D}b^*_{[Z_1 \cap Z_2]}(X)
\to 
{\mathcal S}{\mathcal D}b^*_{[Z_1]}(X) \oplus
{\mathcal S}{\mathcal D}b^*_{[Z_2]}(X) 
\to
{\mathcal S}{\mathcal D}b^*_{[Z_1 \cup Z_2]}(X) \to 0.
$$
Then the injectivity is clear, and the surjectivity 
follows from Lemma \ref{lemma:sheafification}. Using
Lemma {\ref{lemma:dual-outer-extension}} instead of
Lemma {\ref{lemma:outer-extension}},
we can prove the exactness of the middle by
the same argument as in the proof of Theorem \ref{thm:close_mv_swj}.

Note that, for the case $*=(s)$, the corollary can be also proved
by taking the dual of \eqref{eq:swj_closed_mv_surj} since all 
the vector spaces in \eqref{eq:swj_closed_mv_surj} have {\bf{FS}} topologies.
\end{proof}

\subsection{Stratif\mbox{}ied and tempered-stratif\mbox{}ied ultradistributions}

In this subsection, we assume that $X$ is a real analytic manifold with arbitrary dimension.
For $U\subset X$ a subanalytic open set, we define the set of
tempered-stratif\mbox{}ied ultradistributions as
$$
{\mathcal D}b^{*ts}_{X_{sa}}(U) :=
\frac{ {\mathcal S}{\mathcal D}b^*_{[X]}(X)}{ {\mathcal S}{\mathcal D}b^*_{[X \setminus U]}(X)} 
= \frac{ {\mathcal D}b^*(X)}{ {\mathcal S}{\mathcal D}b^*_{[X \setminus U]}(X)} 
\ .
$$
%where $\partial U$ denotes the closed subanalytic subset $\overline{U} \setminus U$.
%For the presheaf ${\mathcal D}b^{*qt}_{X_{sa}}$ on $X_{sa}$ we have:

\begin{theorem}{\label{thm_qt_versus_t}}
Let $U$ be an open subanalytic subset of $X$.
\begin{enumerate}
\item The ring $\Gamma(\overline{U}, {\mathcal D}_X)$ acts on ${\mathcal D}b^{*ts}_{X_{sa}}(U)$.  
\item Let $V$ be an open subanalytic subset of $X$. Then we have the following exact sequence.
$$
{\mathcal D}b^{*ts}_{X_{sa}}(U \cup V) 
\to 
{\mathcal D}b^{*ts}_{X_{sa}}(U)  \oplus
{\mathcal D}b^{*ts}_{X_{sa}}(V) 
\to
{\mathcal D}b^{*ts}_{X_{sa}}(U \cap V) \to 0.
$$
Further, if $\operatorname{dim}_{\rea}(X) \le 2$, then the first morphism of the above sequence
is injective. Hence, in this case,
${\mathcal D}b^{*ts}_{X_{sa}}$ is a subanalytic sheaf
on $X_{sa}$ and a $\rho_! {\mathcal D}_X$ module.
\item If $X \setminus U$ is
  1-regular, then ${\mathcal D}b^{*ts}_{X_{sa}}(U)$ coincides
  with the sections of tempered ultra-distributions of class $*$ on
  $U$, that is,
$$
{\mathcal D}b^{*ts}_{X_{sa}}(U) =
%\{u \in 
\mathcal Db^{*t}_X(U) \ .%;\, \text{$u$ can be extended to an element in ${\mathcal D}b^*(X)$}\}.
$$
\end{enumerate}
\end{theorem}

\begin{proof}
\emph{1.} Let $W \supset \overline{U}$ be an open subset in $X$ and $P \in \Gamma(W,\, \mc D_X)$.
We choose a function $\varphi \in \ms C^*(X)$ with $\operatorname{supp}\varphi \subset W$ and
$\varphi(x)=1$ in a neighborhood of $\overline{U}$. Then $\varphi P$ can be considered as a
differential operator on $X$ with coefficients in $\ms C^*(X)$, and thus, 
it acts on $\mathcal Db^{*t}_X(U)$.
This action does not depend on a choice of $\varphi$. Indeed, this follows from the fact that,
for any $u \in {\mathcal D}b^*(X)$ and $\psi \in \ms C^*(X)$ with 
$\operatorname{supp}\psi \cap \overline{U} = \varnothing$, we have
$\psi u \in {\mathcal S}{\mathcal D}b^*_{[X \setminus U]}(X)$.

\

\noindent
\emph{2.} The exactness is an immediate consequence of the
following commutative diagram whose rows and columns are exact.
$$
\begin{matrix}
&
&  
&
& 0
&
& 0 \\
%%%%%%%%%%%
&
& 
&
& \downarrow
&
& \downarrow \\
%%%%%%%%%%%

&
&
&
& {\mathcal S}{\mathcal D}b^*_{[X \setminus U]}(X) \oplus
{\mathcal S}{\mathcal D}b^*_{[X \setminus V]}(X) 
&\to
&{\mathcal S}{\mathcal D}b^*_{[X \setminus (U \cap V)]}(X) 
&\to 
& 0 \\
%%%%
&
& 
&
& \downarrow
&
& \downarrow \\
%%%%%%%
0 
&\to
&{\mathcal D}b^*(X)
&\to
& {\mathcal D}b^*(X) \oplus
{\mathcal D}b^*(X) 
&\to
&{\mathcal D}b^*(X) 
&\to 
& 0  \\
%%%%%%%%%%
&
& \downarrow
&
& \downarrow
&
& \downarrow \\
%%%%%%%
&
&{\mathcal D}b^{*ts}_{X_{sa}}(U \cup V)
&\to
&{\mathcal D}b^{*ts}_{X_{sa}}(U) \oplus
{\mathcal D}b^{*ts}_{X_{sa}}(V) 
&\to
&{\mathcal D}b^{*ts}_{X_{sa}}(U\cap V) 
& 
& \\
%%%%%%%%%%%%
&
& \downarrow
&
& \downarrow
&
& \downarrow \\
%%%%%%%%%%%%
&
& 0
&
& 0
&
& 0
\end{matrix}
$$
The assertion for the case $\operatorname{dim} X \le 2$ comes from
Corollary {\ref{coro:ultra-exact}}.

\

\noindent
\emph{3.} follows from Corollary {\ref{coro:1-regular-equiv}}.
\end{proof}

\subsection{Higher dimensional case}

Let $\mathcal G$ be a subanalytic sheaf on $X_{sa}$ and denote by
$r_{V,U}$ the restriction map of ${\mathcal G}$ for $V \subset U$ open
subanalytic subsets. Assume that ${\mathcal G}$ satisfies the following
conditions.

\begin{enumerate}
\item If $U\in\Op(\xsa)$ has smooth boundary, then $\mc G(U)\simeq\mc
  Db^{*t}(U)$.
%$$
%{\mathcal G}(U) = \text{the set of usual tempered ultra-distributions
%  along $U$ in $X$}.
%$$
  In particular, ${\mathcal G}(X) = {\mathcal D}b^*(X)$.
\item For any $U\in\Op(\xsa)$, $r_U := r_{U, X}: {\mathcal D}b^*(X) =
  {\mathcal G}(X) \to {\mathcal G}(U)$ is surjective, i.e.  $\mathcal
  G$ is quasi-injective.
\end{enumerate}

Note that, since ${\mathcal G}$ is a sheaf in $X_{sa}$, 
for open subanalytic subsets $U$ and $V$, the sequence
$$
0\to {\mathcal G}(U \cup V) \to
{\mathcal G}(U) \oplus {\mathcal G}(V) \to {\mathcal G}(U\cap V)
$$
is exact.

If $\operatorname{dim}X = 2$, the sheaf ${\mathcal
  D}b^{*ts}_{X_{sa}}$ satisf\mbox{}ies the the conditions 1. and 2. above. Let us
prove that, if $\mathrm{dim}X>2$, then such a subanalytic sheaf $\mc
G$ does not exist.
%However if $\operatorname{dim}X > 2$, we can see that such a $\mathcal G$
%never exists. 
For $Z$ a subanalytic closed subset of $X$, set
$$
{\mathcal F}(Z) := \operatorname{ker} ({\mathcal D}b^*(X) \to 
{\mathcal G}(X \setminus Z)) \subset {\mathcal D}b^*(X) \ .
$$

Then, for any closed subanalytic sets $Z_1 \subset Z_2\subset X$, there
exists an injective morphism 
$$
i_{Z_2, Z_1}: {\mathcal F}(Z_1) \hookrightarrow {\mathcal F}(Z_2)
$$
satisfying $i_{Z_3, Z_1} = i_{Z_3, Z_2} \circ i_{Z_2, Z_1}$. One
checks easily that the sequence
$$
0\to {\mathcal F}(Z_1 \cap Z_2) \to {\mathcal F}(Z_1) \oplus {\mathcal
  F}(Z_2) \to {\mathcal F}(Z_1\cup Z_2) \to 0
$$
is exact. 

Using the fact that $i_{Z_2, Z_1}$ is injective, and $ {\mathcal F}(Z)
\subset {\mathcal F}(X \setminus B) = \Gamma_{X \setminus B}(X,
\,{\mathcal D}b^*) $, for any open ball $B$ with $Z \cap B =
\varnothing$, it is easy to see that
$$
{\mathcal F}(Z) \subset \Gamma_Z(X, {\mathcal D}b^*) \ .
$$
Further, if $Z$ is a smooth manifold we have
$$
{\mathcal F}(Z) = \Gamma_Z(X, {\mathcal D}b^*) \ .
$$

Now let $Z_1,Z_2$ be two smooth hypersurfaces of $X$, from the
commutative diagram with exact rows
$$
\begin{matrix}
0 
&\to 
&{\mathcal F}(Z_1 \cap Z_2) 
&\to
&{\mathcal F}(Z_1) \oplus {\mathcal F}(Z_2) 
&\to 
&{\mathcal F}(Z_1\cup Z_2) 
\\
&
&\text{\rotatebox[origin=c]{270}{$\rightarrowtail$}}
&
&\text{\rotatebox[origin=c]{270}{$\simeq$}}
&
&\text{\rotatebox[origin=c]{270}{$\rightarrowtail$}}
\\
0 
&\to 
&\Gamma_{Z_1 \cap Z_2}(X,\, {\mathcal D}b^*)
&\to
& \Gamma_{Z_1}(X,\, {\mathcal D}b^*)
\oplus 
\Gamma_{Z_2}(X,\, {\mathcal D}b^*)
&\to 
&
\Gamma_{Z_1 \cup Z_2}(X,\, {\mathcal D}b^*) \ ,
\end{matrix}
$$
one obtains that
$$
{\mathcal F}(Z_1 \cap Z_2) = \Gamma_{Z_1\cap Z_2}(X,\, {\mathcal D}b^*) \ .
$$

If $\operatorname{dim}(X) > 2$, then we can f\mbox{}ind a pair of smooth
hypersurfaces $Z_1$ and $Z_2$ such that $Z_1 \cap Z_2$ consists of two
smooth curves $W_1$ and $W_2$ tangentially intersecting at $p$. For
example, let $X=\mathbb R^3$ with coordinates $(x,y,z)$, $Z_1 = \{z =
0\}$ and $Z_2 = \{z + y^2 - x^{2m} = 0\}$.  Then, since $Z_1 \cap Z_2
= W_1 \cup W_2$, we have the following exact sequence
\begin{equation}
  \label{eq:contradiction}
\begin{aligned}
\Gamma_{W_1}(X,\, {\mathcal D}b^*)
\oplus 
\Gamma_{W_2}(X,\, {\mathcal D}b^*)
&={\mathcal F}(W_1) \oplus {\mathcal F}({W_2}) \\
&\to
{\mathcal F}(W_1 \cup W_2) = \Gamma_{W_1 \cup W_2}(X,\, {\mathcal D}b^*)
\to 0.
\end{aligned}
\end{equation}

This gives a contradiction. Indeed, the exactness of
\eqref{eq:contradiction} implies that any ultra-distributions
supported on $W_1 \cup W_2$ is the sum of ultra-distributions
supported in $W_1$ or $W_2$. When $W_1$ and $W_2$ are tangent at
$p$, this is not true.

\appendix
\section{Super growth indicators}\label{section:sgi}
%\lhead{sgi}
%\input{sgi.tex}

The aim of the appendix is to show Proposition {\ref{prop:dense-whitney}}.
To prove the proposition we need several lemmas and propositions. Their proofs
are given only for the case $*=(s)$ in this paper as those for $*=\{s\}$ 
can be done by the similar technique.

\begin{adefs}
  We say that a $\ms C^2$-function $\varphi(t):\rea_{\ge 0} \to \mathbb
  R$ is a \emph{super growth indicator} if 
%the following conditions are
%  satisf\mbox{}ied:
%\begin{enumerate}
%\item 
$\displaystyle\lim_{t\to\infty} \varphi'(t) = \infty$ and 
%\item 
$\varphi''(t) \ge 0$.  
%\end{enumerate}
Further, a function $\varphi(t):\rea_{\ge 0} \to \mathbb R$
is said to be a \emph{linear growth indicator} if there exist a positive
constant $h > 0$ and a constant $C$ such that $\varphi(t) = ht + C$.
\end{adefs}
\noindent
Note that a super growth indicator is a convex function.

Let $\varphi_1(t),\varphi_2(t)$ be two super growth indicators, we
write $\varphi_1 \ll \varphi_2$ if and only if there exists $C \in
\mathbb R$ such that, for any $t\in\rea_{\geq0},$ $\varphi_1(t) \le
\varphi_2(t) + C$.  %We can easily obtain the following lemma:

\begin{alemma}{\label{lemma:super-basic}}
\begin{enumerate}
\item Let $\varphi(x)$ be a super growth indicator. There exists
  $\{\gamma_k\}_{k\in\integergz}\subset\rea$ such that, for any
  $k\in\integergz$,
\begin{equation}{\label{eq:super-growth-reverse}}
\varphi(t) \ge kt + \gamma_k \ .
\end{equation}
\item For any $\{\gamma_k\}_{k\in\integergz}\subset\rea$, there exists
  a super growth indicator $\varphi(t)$ such that
\begin{equation}{\label{eq:super-growth}}
  \underset{k\ge 1}{\sup} \left( kt + \gamma_k \right) \ge \varphi(t).
\end{equation}
\end{enumerate}
\end{alemma}

\begin{proof}
  Easy.
\end{proof}

Let $X$ be a real analytic manifold and $A$ a closed subanalytic set, and
let $A_\alpha$ be a stratum of a stratif\mbox{}ication $\{A_\alpha\}$ of $A$.

We set
$$
B(A_\alpha) := \underset{A_\alpha \subset \bar{A}_\beta}{\cup} \bar{A}_\beta \ .
$$

Clearly $B(A_\alpha)$ is a closed subanalytic subset of $X$.

\begin{aprop}{\label{prop:decay-estimate}}
  Let $X=\mathbb R^n$ and $A\subset X$ a closed subanalytic set with a
  1-regular stratif\mbox{}ication $\{A_\alpha\}$. 
  Assume $*=(s)$ $($resp. $*=\{s\})$. Then 
  for any $\alpha$ and any $F\in {\mathcal S}{\mathcal W}^{*}_A(X)$ with
  $F\vert_{\bar{A}_\alpha} = 0$, 
  there exists a super $($resp. linear$)$ growth indicator
  $\varphi_\alpha$ satisfying the following condition. For any $h >
  0$ $($resp. some $h > 0)$, there exists a constant $C_h$ such that, for any $ \epsilon >
  0$,
$$
\vert\vert F \vert\vert_{B(A_\alpha,\, \epsilon), s, h}
\le C_h \exp(-\varphi_\alpha(\epsilon^{-\sigma})) \ ,
$$
where $B(A_\alpha,\,\epsilon) = \{x \in B(A_\alpha);\, 
\operatorname{dist}(x, \bar{A}_\alpha) \le \epsilon\}$
and $\sigma = \frac{1}{s-1}$.
\end{aprop}
\begin{proof}
  We may assume that $A = B(A_\alpha)$, in particular, the number of
  the strata is f\mbox{}inite.  Since the closure of each stratum
  $A_\beta$ is 1-regular, connected and compact, we can f\mbox{}ind a
  constant $\kappa_\beta$ such that for any $x,y\in\bar{A}_\beta$
  there exists a subanalytic curve $l\subset\bar{A}_\beta$ joining $x$
  and $y$ such that
$$
\vert l \vert \le \kappa_\beta \vert x - y \vert.
$$
We set
$$
\kappa = \max \{\kappa_\beta\}.
$$
Let $y$ be a point in some stratum $A_\beta$. As $\bar A_\alpha$ is
compact, there exists $x\in\bar A_\alpha$ such that
%Since $\bar{A}_\alpha \subset \bar{A}_\beta$,
%we f\mbox{}ind a point $x\in \bar{A}_\beta \cap \bar{A}_\alpha$ such that
$$
\operatorname{dist}(y, \bar{A}_\alpha) = \vert y - x \vert
$$
As $\bar A_\alpha\subset\bar A_\beta$, 
%Hence 
there exists a curve $l$ in $\bar{A}_\beta$ joining $x$ and $y$ such that 
$$
\vert l \vert \le \kappa \vert x - y \vert \ .
$$

Now, let $F=\{f_\alpha\}_\alpha\in {\mathcal S}{\mathcal
  W}^{(s)}_A(X)$ such that $F\vert_{\bar{A}_\alpha} = 0$. We have
that, for any $x\in\bar A_\alpha,y\in\bar A_\beta$,
$f_\gamma(y)=R_m(S_\gamma F; y, x)$. Hence,
%It follows
from Lemma \ref{lemma:Whitney-curve}, it follows that
$$
\begin{aligned}
\vert f_\gamma(y)| = |R_m(S_\gamma F;\,y,\,x) \vert &\le 
\frac{(\sqrt{n}\vert l \vert)^{m+1}}{m!} \max_{\vert\eta\vert = \vert\gamma\vert + m+1}
\sup_{y \in l} \vert f_\eta(y) \vert \\
&\le
\frac{(\sqrt{n}\kappa\vert x - y \vert)^{m+1}}{m!} 
(\vert\gamma\vert + m + 1)!^s h^{\vert \gamma \vert + m + 1} 
\vert\vert F \vert\vert_{A, s, h} \\
&\le (2^{s+1}\sqrt{n}\kappa h\vert x - y \vert)^{m+1}{m!}^{s-1}
(\vert\gamma\vert)!^s (2^sh)^{\vert \gamma \vert} \vert\vert F \vert\vert_{A, s, h}.
\end{aligned}
$$
Hence for any $h'>0$ and $h>0$
$$
\frac{\vert f_\gamma(y) \vert}{\vert\gamma\vert!^s (2^sh')^{\vert\gamma\vert}}
\le \left(2^{s+1}\sqrt{n}\kappa h\vert x - y \vert\right)^{m+1}{m!}^{s-1}
\left(\frac{h}{h'}\right)^{\vert \gamma \vert} \vert\vert F \vert\vert_{A, s, h}.
$$

Since 
$$
\inf_{m \in \mathbb N} t^{m+1}{m!}^{s-1} \le C\exp\bigr{-Bt^{-\sigma}}
$$
for some positive constants $B$ and $C$, we obtained
$$
\frac{\vert f_\gamma(y) \vert}{\vert\gamma\vert!^s (2^sh')^{\vert\gamma\vert}}
\le C\vert\vert F \vert\vert_{A,s,h} \left(\frac{h}{h'}\right)^{\vert\gamma\vert}
\exp \left(-B'\left(\frac{1}{h}\right)^\sigma(\vert x - y\vert)^{-\sigma}\right)
$$
for a constant $B'>0$. 

By Lemma \ref{lemma:super-basic}, there exists a super growth
indicator $\varphi$ such that
$$
\sup_{h > 0}
\frac{1}{C\vert\vert F \vert\vert_{A,s,h}} 
\exp \left(B'\left(\frac{1}{h}\right)^\sigma t \right)
\ge \exp(\varphi(t)).
$$
Then, one checks easily that, for any $h' > 0$, there exists 
a constant $C_{h'} > 0$ such that:
$$
\sup_{h' > h > 0}
\frac{1}{C\vert\vert F \vert\vert_{A,s,h}} 
\exp \left(B'\left(\frac{1}{h}\right)^\sigma t \right)
\ge \frac{1}{C_{h'}}\exp(\varphi(t)).
$$
Therefore, we obtain
$$
\begin{aligned}
\frac{\vert f_\gamma(y) \vert}{\vert\gamma\vert!^s (2^sh')^{\vert\gamma\vert}}
&\le \inf_{0 < h < h'}
C\vert\vert F \vert\vert_{A,s,h} \left(\frac{h}{h'}\right)^{\vert\gamma\vert}
\exp \left(-B'\left(\frac{1}{h}\right)^\sigma(\vert x - y\vert)^{-\sigma}\right) \\
&\le \inf_{0 < h < h'}
C\vert\vert F \vert\vert_{A,s,h} 
\exp \left(-B'\left(\frac{1}{h}\right)^\sigma(\vert x - y\vert)^{-\sigma}\right) \\
&\le {C_{h'}}\exp(-\varphi(\vert x - y \vert^{-\sigma})).
\end{aligned}
$$
This entails the result.
\end{proof}

\begin{alemma}{\label{lemma:super-growth-exp}}
  For any constant $C>1$, $\epsilon > 0$ and any super growth
  indicator $\psi$, there exists a super growth indicator
  $\varphi$ satisfying the following conditions.
\begin{enumerate}
\item $\varphi(t) \ll \psi(t)$.
\item $\varphi'(0) > 0$ and $\varphi(0) = 0$.
\item for any $s,t \in [1,\infty)$, $\varphi(st) \le C s^{1+\epsilon}\varphi(t)$.
\end{enumerate}
\end{alemma}

\begin{proof}
Note that the third condition is equivalent to

$$
\varphi(s) \le C \left(\frac{s}{t}\right)^\epsilon
\frac{\varphi(t)}{t}s\qquad s \ge t \ge 1 \ .
$$

Set 
$$
	\gamma = C^{\frac{1}{5}} > 1 \ .
$$

Without loss of generality, we assume that $\psi(0) = 0$ and
$\psi'(0) = M$, for some constant $M>1$.  Let
$\{\hat{x}_k\}_{k\in\integergz} \subset [0,\infty)$ satisfy:
\begin{enumerate}
\item $0=\hat{x}_0 < \hat{x}_1 < \hat{x}_2 < \dots$ and 
$\displaystyle\lim_{k\to\infty} \hat{x}_k = \infty$.
\item $\psi'(\hat{x}_k) = M\gamma^{k}$.
\end{enumerate}

Given an increasing sequences $\{x_k\}_{k\in\integer{\geq0}}\subset\rea$ with $x_0 = 0$ and
$\displaystyle\lim_{k\to\infty} x_k = \infty$, we set
$$
g(x) := \sum_{i=0}^{k-1} \gamma^i (x_{i+1} - x_i) + \gamma^k(x - x_k)\qquad
x \in [x_k, x_{k+1}].
$$

Let $\{x_k\}_{k\in\integer{\geq0}}$ be an increasing sequence of real
numbers satisfying the following conditions:
\begin{enumerate}
\item $x_0 = 0$ and $x_1 \ge 2$,
\item $x_{k+1}  \ge \max \{\gamma^{\frac{1}{\epsilon}}x_k,\, \hat{x}_{k+1}\}$ for 
$k\in\nat$,
\item $g(x_{k+1}) \ge \gamma^{k-1} x_{k+1}$ for $k\in\nat$.
\end{enumerate}

Note that such a sequence $\{x_k\}_{k}$ always exists as, for f\mbox{}ixed
$k$, we have
$$
\lim_{x \to \infty}
\frac{\displaystyle\sum_{i=0}^{k-1} \gamma^i (x_{i+1} - x_i) + \gamma^k(x - x_k)}{x} = \gamma^k.
$$

Clearly, if $t\leq s$, then $\bigr{\frac st}^\ep\geq1>\gamma^{-2}$.
More precisely, for $0 < t \le s$ with $t \in [x_l, x_{l+1}]$ and $s
\in [x_k, x_{k+1}]$ ($l<k$), we easily obtain
$$
\left(\frac{s}{t}\right)^\epsilon 
\ge
\left(\frac{x_k}{x_{l+1}}\right)^\epsilon 
\ge \gamma^{k-l-2}.
$$

Moreover, since the function $g(t)$ is convex with $g(0) = 0$, the
function $\frac{g(t)}{t}$ is an increasing function of $t$. Hence,
for any $t \in [x_l, x_{l+1}]$ ($1 \le l$), we have
$$
\frac{g(t)}{t} \ge \frac{g(x_l)}{x_l} \ge \gamma^{l-2}.
$$

It follows that, for $0 < t \le s$, $s\in [x_k, x_{k+1}]$,
$$
\gamma^k =
\gamma^4 \gamma^{k-l-2} \gamma^{l-2} \le
\frac{1}{\gamma} C \left(\frac{s}{t}\right)^\epsilon \frac{g(t)}{t},
$$
and 
$$
g(s) = \sum_{i=0}^{k-1} \gamma^i (x_{i+1} - x_i) + \gamma^k(s - x_k)
\le \gamma^k s \le \frac{1}{\gamma} C
\left(\frac{s}{t}\right)^\epsilon \frac{g(t)}{t} s \ .
$$
%for $0 < t \le s$.

Since $\psi(s)$ is an increasing convex function, we have
$$
\begin{aligned}
g(s) &= \sum_{i=0}^{k-1} \gamma^i (x_{i+1} - x_i) + \gamma^k(s - x_k) \\
&=
\sum_{i=0}^{k-1} \frac{\psi'(\hat{x}_{i})}{M} (x_{i+1} - x_i) + 
\frac{\psi'(\hat{x}_{k})}{M}(s - x_k) \\
&\le
\frac{1}{M} \left(\sum_{i=0}^{k-1} \psi'(x_i)(x_{i+1} - x_i) + 
\psi'(x_k)(s - x_k)\right) \\
&\le \frac{1}{M}\psi(s) \ .
\end{aligned}
$$

Hence, the continuous convex increasing function $g(s)$ satisf\mbox{}ies:
$$
g(st) \le \frac{1}{\gamma} C s^{1+\epsilon}g(t) \qquad
\text{for } s \in [1,\infty)\text{ and } t \in [0,\infty)
$$
and
$$
g(t) \le \frac{1}{M}\psi(t).
$$

Let $\chi_h(x)$ be a non-negative $\cinfty(\rea)$ function such that
$$
\operatorname{supp}(\chi_h(x)) \subset \{\vert x \vert \le h\}
\qquad \text{and }
\int_{\mathbb R} \chi_h(x) dx = 1.
$$

For $0 < h < 1$, set
$$
\varphi_h(t) = 
\left\{
\begin{array}{lc}
\ t & 0 \le t \le 1 \\
\displaystyle\int_{\rea} g(t-x) \chi_h(x) dx \qquad & t \ge 1
\end{array}
\right.
$$

Then we have $\varphi_h \in\cinfty(\rea)$ and
$$
g(t) \le \varphi_h(t) \le g(t+h).
$$

Hence, we obtain
\begin{equation}
  \label{eq:phih}
\varphi_h(st -h)
\le
g(st)
\le
\frac{1}{\gamma}C s^{1+\epsilon} g(t)
\le
\frac{1}{\gamma}C s^{1+\epsilon} \varphi_h(t).
\end{equation}
Now, we replace $s$ by $(1+h)s$ into \eqref{eq:phih}. Then we get, for
suf\mbox{}f\mbox{}iciently small $h>0$ and $s,t \in [1,\infty)$,
$$
\varphi_h(st) \le
\varphi_h((1+h)st - h) \le \frac{(1+h)^{1+\epsilon}}{\gamma}C s^{1+\epsilon}\varphi_h(t)
\le C s^{1+\epsilon}\varphi_h(t) \ .
$$

Further, it follows from $g(t) \le \frac{1}{M}\psi(t)$ that for
suf\mbox{}f\mbox{}iciently small $h>0$ and $t \ge 1$ we have
$$
\varphi_h(t) \le \psi(t).
$$

The conclusion follows.
\end{proof}

\begin{alemma}{\label{lemma:holmorphic-super}}
Let $\sigma > 0$. For any super growth indicator $\varphi(t)$,
there exist a holomorphic function $p(\xi)$ on $\mathbb C \setminus \rea_{\le0}$ and a super growth 
indicator $\rho(t)$ satisfying the following conditions.
\begin{enumerate}
\item There exists $C>0$ such that, for any $\xi \in \mathbb C \setminus \rea_{\le0}$,
\begin{equation}{\label{eq:entire-estimate-1}}
\vert p(\xi) \vert \le 
C\exp\left(\varphi(\vert \xi \vert ^\sigma)\right) 
\end{equation}
holds. Moreover $p(\xi)$ is real valued for
$\xi \in \mathbb R_{>0}$.
\item The inequality
\begin{equation}{\label{eq:entire-estimate-2}}
\exp\left(\rho(\vert \xi \vert^\sigma)\right) \le 
\vert p(\xi) \vert 
\end{equation}
holds for any
$
\xi \in \left\{\xi = re^{i\theta} \in \mathbb C;\, r > 0,\, 
\vert \theta \vert < \min\left\{ \frac{\pi}{4},\, \frac{\pi}{8\sigma}\right\}\right\}.
$
\end{enumerate}
\end{alemma}

\begin{proof}
  It is enough to construct an entire function $p(\xi)$ for $\sigma = \frac{1}{2}$ 
that satisfies the estimate \eqref{eq:entire-estimate-1} for any $\xi \in \mathbb C$ and
{\eqref{eq:entire-estimate-2}} on
$
\left\{re^{i\theta} \in \mathbb C;\, r > 0,\, 
\vert \theta \vert < \frac{\pi}{4}\right\}$.
Then, for an arbitrary $\sigma > 0$, the holomorphic function $p(\xi^{2\sigma})$ 
on $\mathbb C \setminus \rea_{\le 0}$ gives a required one.

We suppose $\sigma = \frac{1}{2}$ in what follows.  Let
  $\psi(t)$ be a super growth indicator satisfying the conditions 2.
  and 3. in Lemma \ref{lemma:super-growth-exp} for some $\epsilon > 1$
  and $C >0$ which will be determined later on.

%For any sequence $\{\xi_k\}_{k\in\integergz}$ such that 
%$0 < \xi_1 < \xi_2 < \dots$ and 
%$\displaystyle\lim_{k\to \infty} \xi_k = \infty$ and 

%Noticing 
%$$
%t\frac{d}{dt}\psi(t^\sigma) = \sigma t^\sigma \psi'(t^\sigma)
%$$
For $s \ge 0$, we set
$$
g(s) := \sigma s\psi'(s).
$$

As
$$
g'(s) = \sigma\left(\psi'(s) + s\psi''(s)\right) \ge \sigma\psi'(0) > 0,
$$
we have that $g(s)$ is a strictly increasing function and 
$\displaystyle\lim_{s\to\infty} g(s) = \infty$.

Now, let $\{\xi_k\}$ be a sequence such that
$$
k = g(\xi_k^\sigma)\qquad k=1,2,\dots
$$
holds for any $k\in\integergz$. Then, since $g(s)$ is strictly
increasing, we have
$$
\xi_1 < \xi_2 < \dots,\qquad\text{and } \lim_{k\to\infty} \xi_k = \infty.
$$

For any $t>0$, we set
\begin{eqnarray*}
n(t) &:=&
\{\text{the number of $\xi_k$ such that $\vert \xi_k \vert \le t$}\}\\
N(t)& := & \int_0^t \frac{n(\lambda)}{\lambda} d\lambda.
\end{eqnarray*}

Moreover, for any $\xi_k \le s < \xi_{k+1}$ we have
$$
 n(s) = n(\xi_k) = k = g(\xi_k^\sigma) \le g(s^\sigma),
$$
hence
$$
\begin{aligned}
N(s) &\le \int_0^s \frac{n(t)}{t} dt \le \int_0^s \frac{g(t^\sigma)}{t} dt \\
     &= \int_0^s \frac{d}{dt}\psi(t^\sigma) dt = \psi(s^\sigma) - \psi(0) = \psi(s^\sigma)
\end{aligned}
$$

Set $$p(\xi):= \Pi_{k\ge 1}\left(1 + \frac{\xi}{\xi_k}\right)$$ By the
Lindel{\"o}f Theorem (see proof of Proposition 4.6, page 59 of
\cite{komatsu_ultradistributionsI}) we conclude that $p(\xi)$ is
absolutely convergent in $\mathbb C$.  The same theorem give the
estimate
$$
\log \sup_{\vert\xi\vert = t}\vert p(\xi) \vert \le 
\int_0^\infty \frac{tN(\lambda)}{(t+\lambda)^2} d\lambda \ .
$$
Thus we get
\begin{eqnarray}\label{eq:logsup}
\log \sup_{\vert\xi\vert = t}\vert p(\xi) \vert &\le &
\int_0^\infty \frac{t\psi(\lambda^\sigma)}{(t+\lambda)^2} d\lambda \notag \\
&=&\int_0^t \frac{t\psi(\lambda^\sigma)}{(t+\lambda)^2} d\lambda +
\int_t^\infty \frac{t\psi(\lambda^\sigma)}{(t+\lambda)^2} d\lambda.
\end{eqnarray}

The f\mbox{}irst term of  right hand side of \eqref{eq:logsup} satisf\mbox{}ies
$$
\int_0^t \frac{t\psi(\lambda^\sigma)}{(t+\lambda)^2} d\lambda 
\le \frac{1}{t}\int_0^t \psi(\lambda^\sigma) d\lambda \le \psi(t^\sigma).
$$
We estimate the second term of the right hand side of
\eqref{eq:logsup} as follows.  Let $\psi$ be a super growth indicator
satisfying the condition 3. in Lemma \ref{lemma:super-growth-exp} for
$C=2$ and $\epsilon = \displaystyle\frac{1}{\sqrt{\sigma}} - 1$.  Then
we have
$$
\begin{aligned}
\int_t^\infty \frac{t\psi(\lambda^\sigma)}{(t+\lambda)^2} d\lambda
&= \int_1^\infty \frac{\psi(t^\sigma\lambda^\sigma)}{(1+\lambda)^2} d\lambda \\
&\le 2\int_1^\infty 
\frac{\lambda^{\sigma(1+\epsilon)}\psi(t^\sigma)}
{1+\lambda^2} d\lambda
\le 2\left(\int_1^\infty \frac{1}{\lambda^{2-\sqrt{\sigma}}} d\lambda\right) \psi(t^\sigma).
\end{aligned}
$$

Hence we obtain
$$
\log \sup_{\vert\xi\vert = t}\vert p(\xi) \vert \le C_\sigma \psi(t^\sigma)
$$
where $C_\sigma$ depends only on $\sigma$.  Therefore $p(\xi)$
satisf\mbox{}ies the condition 1. of Lemma
\ref{lemma:holmorphic-super} when $\psi(t)$ is a super growth
indicator given by Lemma \ref{lemma:super-growth-exp} with the
indicator $\displaystyle\frac{\varphi(t)}{C_\sigma}$ and the constants
$C=2$ and $\epsilon = \displaystyle\frac{1}{\sqrt{\sigma}} - 1$.

Now, for $k\in\integergz$, set
$$
l_k := \frac{k^\frac{1}{\sigma}}{\xi_k} \ .
$$
Let us prove that, for $k \to \infty$, $l_k \to \infty$.  Indeed,
unless $l_k \to \infty$, then there exists $M > 0$ and an increasing
sequence of natural numbers $\{k_p\}$ with $l_{k_p} \le M$. Since
$$
k_p = g(\xi_{k_p}^\sigma) = \sigma \xi_{k_p}^\sigma\psi'(\xi_{k_p}^\sigma) 
=
\sigma \frac{k_p}{l_{k_p}^\sigma} \psi'\left(\frac{k_p}{l_{k_p}^\sigma}\right).
$$
holds, we have
\begin{equation}
  \label{eq:lsigmakp}
l_{k_p}^\sigma = \sigma \psi'\left(\frac{k_p}{l_{k_p}^\sigma}\right)  \ .
\end{equation}
Since $\psi'(t) \to \infty$ ($t \to \infty$) and
$\frac{k_p}{l_{k_p}^\sigma} \to \infty$ ($p \to \infty$), the right
hand side of \eqref{eq:lsigmakp} tends to $\infty$.  This contradicts
to the fact that the left hand side of \eqref{eq:lsigmakp} is bounded.

\

Now, $p(\xi)$ can be written in the form:
$$
p(\xi) = \Pi_{k\ge 1}\left(1 + \frac{l_k\xi}{k^{\frac{1}{\sigma}}}\right).
$$
If $\xi \in D = \{\vert \Im \xi \vert \le \Re \xi\}$, we have
$$
\left| 1 + \frac{l_k \xi}{k^{\frac{1}{\sigma}}}\right| \ge 
1 + \frac{l_k \Re \xi}{k^{\frac{1}{\sigma}}} \ge
1 + \frac{l_k \vert \xi \vert}{\sqrt{2}k^{\frac{1}{\sigma}}} \ .
$$
Since $l_k \to \infty$, for any given $L > 0$ there exists $k_0$ such
that $l_k \ge \sqrt{2}L$ ($k \ge k_0$).  Hence, for $\xi \in D$, we
get
$$
\vert p(\xi) \vert \ge
\displaystyle
\frac{\left| \Pi_{k\le k_0}\left(1 + \frac{l_k\xi}{k^{\frac{1}{\sigma}}}\right)\right|}
{\Pi_{k\le k_0}\left(1 + \frac{L \vert \xi \vert}{k^{\frac{1}{\sigma}}}\right)}
\Pi_{1 \le k }\left(1 + \frac{L\vert \xi \vert}{k^{\frac{1}{\sigma}}}\right)
\ge
C_L\Pi_{1 \le k }\left(1 + \frac{L\vert \xi \vert}{k^{\frac{1}{\sigma}}}\right)
$$
for some $C_L > 0$. It is well know that there exist $A,B>0$ such
that, for $t \ge 0$,
$$
\Pi_{1 \le k }\left(1 + \frac{t}{k^{\frac{1}{\sigma}}}\right) \ge
A \exp (B t^\sigma) \ .
$$
%holds for some positive constants $A>0$ and $B > 0$. 
Therefore we have that for any $L > 0$
$$
\vert p(\xi) \vert \ge AC_L \exp \left(BL^\sigma \vert \xi \vert^\sigma\right)
=  \exp \left(BL^\sigma \vert \xi \vert^\sigma + \log(AC_L)\right) \qquad
\xi \in D.
$$
This implies that there exists a super growth indicator $\rho(t)$
satisfying
$$
\vert p(\xi) \vert \ge \exp \rho(\vert \xi \vert^\sigma) \qquad \xi \in D.
$$
\end{proof}

\begin{alemma}{\label{lemma:partition-of-unity}}
  Let $s > 1$.  Assume $*=(s)$ $($resp. $*=\{s\})$. 
  Then for any compact set $K\subset\rea^n$ and 
  for any super $($resp. any linear$)$ growth indicator $\varphi$, 
  there exists a family of functions
  $\{\chi_\epsilon(x)\}_{\epsilon > 0}\subset \ms C^* (\rea^n)$
  satisfying the following conditions.
\begin{enumerate}
\item For any $h>0$ $($resp. some $h > 0)$ there exists a constant $C_h$ such that
$$
\vert\vert \chi_\epsilon(x) \vert\vert_{\mathbb R^n, s,h} 
\le C_h \exp\left(\varphi\left(\epsilon^{-\sigma}\right)\right)\qquad
\text{for any } \epsilon > 0.
$$
\item $\operatorname{supp}(\chi_\epsilon(x)) \subset K_{\epsilon}$.
\item $\chi_\epsilon(x) \ge 0$ and, if $x \in K_{\frac{\epsilon}{2}}$,
  $\chi_\epsilon(x) = 1$.
\end{enumerate}
Where $K_{\epsilon}$ denotes the set $\{x\in X;\,
\operatorname{dist}(x,K)\le \epsilon\}$ and $\sigma = \frac{1}{s-1}$.
\end{alemma}

\begin{proof}
  Let $\psi$ be a super growth indicator, $p(\xi)$ a holomorphic
  function as given in Lemma \ref{lemma:holmorphic-super} with the
  super growth indicator $\psi(t)$ and $\sigma > 0$.  We set
$$
f(x) =
\left\{
\begin{matrix}
\displaystyle\frac{1}{p\left(x^{-1}\right)} \qquad & x > 0\ , \\
\\
0 & x \le 0 \ .
\end{matrix}
\right.
$$
One checks easily that $f(z)$ is holomorphic in the sector
$$
S = \{z \in \mathbb C; \vert \operatorname{arg}(z) \vert < \kappa\} \ ,
$$
for some $\kappa > 0$, and that there exists a super growth indicator
$\rho$ such that
$$
\vert f(z) \vert < \exp (-\rho(\vert z \vert^{-\sigma}))
$$
holds for $z \in S$. Hence, using the Cauchy inequality, it is easy to
see that $f(x) \in {\ms C}^{(s)}(\rea)$, that is, for $h>0$ and any
compact set $K \subset \mathbb R$ we have
$$
\vert\vert f(x) \vert\vert_{K, s,h} < \infty \ .
$$
Set 
$$
g_\epsilon(x):= f(x+2\epsilon)f(-x+2\epsilon) \ .
$$
Then, for $0 < \epsilon < 1$, we have
$$
\vert\vert g_\epsilon(x) \vert\vert_{\mathbb R, s, h}
\le \vert\vert f(x) \vert\vert^2_{[0,4],s, \frac{h}{2}}
$$
and $\operatorname{supp}(g_\epsilon(x)) \subset \{\vert x \vert \le
2\epsilon\}$.  Moreover, there exists $M>0$ such that, for any $x>0$,
$$
\vert f(x) \vert \ge M\exp(-\psi(\vert x \vert^{-\sigma})) \ .
$$
Hence, we get
$$
g_\epsilon(x) \ge M\exp(-2\psi( \epsilon^{-\sigma})) \ ,
$$
for $x \in [-\epsilon, \epsilon]$. Thus
$$
\int_{\mathbb R} g_\epsilon(x) dx 
\ge 2\epsilon M\exp(-2\psi(\epsilon^{-\sigma})) \ .
$$
Now, set
$$
\hat{g}_\epsilon(x) := \frac{g_\epsilon(x)}{\displaystyle\int_{\mathbb
    R} g_\epsilon(x) dx} \ .
$$
Then $\hat{g}_\epsilon$ satisf\mbox{}ies
$$
\vert\vert \hat{g}_\epsilon \vert\vert_{\mathbb R, s, h}
\le 
\vert\vert f \vert\vert^2_{[0,4],s, \frac{h}{2}}\frac{1}{2\epsilon M}
\exp(2\psi( \epsilon ^{-\sigma}))
\le M_h \exp(3\psi( \epsilon^{-\sigma}))
$$
for some constant $M_h$ that depends only on $h$.
Set
$$
\hat{g}_\epsilon(x_1,\dots,x_n) := 
\hat{g}_\epsilon(x_1)\hat{g}_\epsilon(x_2)\dots \hat{g}_\epsilon(x_n).
$$
Then $\hat{g}_\epsilon(x)$ satisf\mbox{}ies the followings conditions.
\begin{enumerate}
\item For any $h > 0$ there exists $M_h >0$ such that
$$
\vert\vert \hat{g}_\epsilon(x_1,\dots,x_n) \vert\vert_{\mathbb R^n,s,h}
\le M_h \exp(3n\psi( \epsilon^{-\sigma})) \qquad\text{for any } \epsilon > 0.
$$
\item 
$$
\int_{\mathbb R^n} \hat{g}_\epsilon(x) dx = 1
$$
\item 
$$
\operatorname{supp}(\hat{g}_\epsilon(x)) \subset \{x \in \mathbb R^n;
\vert x \vert \le 2\sqrt{n}\epsilon\}.
$$
\end{enumerate}
Let $\hat{\chi}_{K_{\frac{3\epsilon}{4}}}$ denote
the characteristic function of $K_{\frac{3\epsilon}{4}}$.
If we chose
$$
\psi(t) = \frac{1}{3n}\varphi( (8\sqrt{n})^{-\sigma}t),
$$ 
then
$$
\chi_{\epsilon} := 
\hat{g}_{\frac{\epsilon}{8\sqrt{n}}} * \hat{\chi}_{K_{\frac{3\epsilon}{4}}}
$$
satisf\mbox{}ies the required conditions. 
\end{proof}

\begin{aprop}{\label{prop:dense-whitney}}
  Let $X$ be a real analytic manifold and $A\subset X$ a compact
  subanalytic set.  Then $\Swj(X)$ is a dense subset of the locally
  convex topological vector space $\Sswj(X)$.
\end{aprop}

\begin{proof}
  Take a 1-regular stratif\mbox{}ication $\{A_\alpha\}$ of $A$ and
  f\mbox{}ix it.  Let $A_\alpha$ be a stratum. Set
$$
Z(A_\alpha) := \underset{\bar{A}_\alpha \cap \bar{A}_\beta \ne \varnothing}
\cup \bar{A}_\beta.
$$
and
$$
U(A_\alpha) := \underset{A_\alpha \subset \bar{A}_\beta}
\cup A_\beta.
$$
Then $Z(A_\alpha)$ is compact and it is a closed neighborhood of
$\bar{A}_\alpha$ in $A$.  Hence by the partition of unity, we may
assume that $X=\mathbb R^n$ and $A$ is a compact subanalytic set.

\ 
 
Let $\Lambda_k$ (resp. $\Lambda_{\le k}$) be the set of indices
$\alpha \in \Lambda$ such that $\dim(A_\alpha) = k$ (resp.
$\dim(A_\alpha) \le k$).  For $\epsilon >
0$, let us determine a family of closed sets
$\{W_{\epsilon,\alpha}\}_{\alpha\in \Lambda}$ in $X$ and a family of
positive constants $\{l_{\epsilon, \alpha}\}_{\alpha \in \Lambda}$ in
the following way.

\noindent
If $k=0$, then for any $\alpha \in \Lambda_0$,
\begin{enumerate}
\item $W_{\epsilon, \alpha} = A_\alpha$ \ ,
\item let $l_{\ep,\alpha}>0$ be such that $\epsilon > l_{\epsilon,
    \alpha}$ and $(W_{\epsilon, \alpha,\,l_{\epsilon, \alpha}} \cap A)
  \subset U(A_\alpha)$. Here $W_{\epsilon, \alpha,\,l_{\epsilon,
      \alpha}} =\{x \in X;\, \operatorname{dist}(x, W_{\epsilon,
    \alpha}) \le l_{\epsilon, \alpha}\}$.
\end{enumerate}

\noindent
Suppose that we already determined $W_{\epsilon, \alpha}$ and
$l_{\epsilon, \alpha}$ for every $\alpha \in \Lambda_{\le k-1}$.  Set
$$
\epsilon_{k-1} = \displaystyle\min_{\alpha \in \Lambda_{\le k-1}} l_{\epsilon, \alpha} \ .
$$
First, let $\{W_{\epsilon, \alpha}\}_{\alpha \in \Lambda_k}$ be such that
$$
\bar{A}_\alpha \setminus 
\left(\underset{A_\beta \subset \bar{A}_\alpha,\, \beta \in \Lambda_{\le k-1}}{\cup} 
W_{\epsilon, \beta, \frac{1}{3}l_{\epsilon, \beta}} \right)
\subset W_{\epsilon, \alpha} \subset A_\alpha.
$$
Note that the set
$$
\underset{A_\beta \subset \bar{A}_\alpha,\, \beta \in \Lambda_{\le k-1}}{\cup} 
W_{\epsilon, \beta, \frac{1}{3}l_{\epsilon, \beta}} 
$$
is a neighborhood of $\bar{A}_\alpha \setminus A_\alpha$.
Since $W_{\epsilon, \alpha} \cap W_{\epsilon, \beta} = \varnothing$ 
for any $\alpha \ne \beta \in \Lambda_k$,
there exists a constant $\delta_{\epsilon, k} > 0$
$$
\delta_{\epsilon, k} = \frac{1}{3}\min_{\alpha\ne\beta \in \Lambda_k} 
\operatorname{dist}(W_{\epsilon, \alpha}, W_{\epsilon, \beta}).
$$
Then, for $\alpha \in \Lambda_k$, let $l_{\epsilon, \alpha}$ satisfy:
\begin{enumerate}
\item $l_{\epsilon, \alpha} < \min \{\delta_{\epsilon,
    k},\,\epsilon_{k-1}\}$ \ ,
\item $(W_{\epsilon, \alpha,\, l_{\epsilon, \alpha}} \cap A) \subset
  U(A_\alpha)$ \ ,
\item for any $\beta \in \Lambda_{\le k-1}$ with $A_{\beta} \cap
  \bar{A}_\alpha = \varnothing$, $W_{\epsilon, \alpha,\, l_{\epsilon,
      \alpha}} \cap W_{\epsilon, \beta,\,l_{\epsilon, \beta}} =
  \varnothing$ holds.
\end{enumerate}
Note that such an $l_{\epsilon, \alpha}$ always exists.  Indeed,
$W_{\epsilon, \alpha} \subset A_\alpha$ is a compact set and
$U(A_\alpha)$ is an open subset in $A$.  Hence, there exists
$l_{\epsilon, \beta}$ satisfying condition 2.  Since $A_\beta \cap
\bar{A}_\alpha = \varnothing$ implies $U(A_\beta) \cap A_\alpha =
\varnothing$, we have by the induction hypothesis
$$
W_{\epsilon, \beta,\,l_{\epsilon, \beta}} \cap W_{\epsilon, \alpha} 
= W_{\epsilon, \beta, l_{\epsilon, \beta}} \cap A \cap W_{\epsilon, \alpha}
\subset U(A_\beta) \cap A_\alpha = \varnothing.
$$

As $W_{\epsilon, \beta,\,l_{\epsilon, \beta}}$ and $W_{\epsilon,
  \alpha}$ are closed sets, the condition 3. can be fullf\mbox{}illed.

\

\noindent
Let $\varphi$ be a super growth indicator. By Lemma
\ref{lemma:partition-of-unity}, for any $h>0$, there exist $C_h>0$ and
a family $\{\chi_{\ep,\beta}\}\subset{\ms C}^{(s)}(X)$
%and let
%$\chi_{\epsilon, \beta} \in {\mathcal E}^{(s)}(X)$ denote the function
%satisfying:
such any for any $\ep, \beta$
\begin{enumerate}
\item[(a)] $\chi_{\epsilon, \beta}(x) = 1$ for 
$x \in W_{\epsilon, \beta, \frac{1}{2}l_{\epsilon, \beta}}$ \ ,
\item[(b)] $\operatorname{supp}{\chi_{\epsilon, \beta}(x)} 
\subset W_{\epsilon, \beta,l_{\epsilon, \beta}}$ \ ,
\item[(c)] %For any $h > 0$, there exists $C_{h}$ that does not depend on 
%$\epsilon$ and $\beta$, and that satisf\mbox{}ies
$
\vert\vert \chi_{\epsilon, \beta} \vert\vert_{\mathbb R^n, s, h} 
\le C_{h}
\exp\left(\varphi(l_{\epsilon,\beta}^{-\sigma})\right).
$
\end{enumerate}
%It follows Lemma 
%\ref{lemma:partition-of-unity}
%that a family of functions $\{\chi_{\epsilon, \beta}\}$ always exists.
Since 
$
\underset{\beta}{\cup} W_{\epsilon, \beta, \frac{1}{2}l_{\epsilon, \beta}}
$ 
is a neighborhood of $A$, we have
\begin{equation}
  \label{eq:prod=0}
\underset{\beta \in \Lambda}{\Pi} (1-\chi_{\epsilon, \beta}(x)) = 0
\end{equation}
in a neighborhood of $A$.

Now, let us prove that, for any indices $\beta_1$ and $\beta_2$ with
$\beta_1 \npreceq \beta_2$ and $\beta_2 \npreceq \beta_1$, we have
\begin{equation}
  \label{eq:prod2=0}
  \chi_{\epsilon, \beta_1}(x)\chi_{\epsilon, \beta_2}(x) \equiv 0 \ .
\end{equation}

Here $\beta_1 \preceq \beta_2$ means $A_{\beta_1} \subset
\bar{A}_{\beta_2}$.  Indeed, if $\operatorname{dim}{A_{\beta_1}} =
\operatorname{dim}{A_{\beta_2}}$ and $\beta_1 \ne \beta_2$, then if
follows from the condition 1 of $l_{\epsilon,\beta}$ and $l_{\epsilon,
  \beta_i} < \frac{1}{3}\operatorname{dist}(W_{\epsilon, \beta_1},
W_{\epsilon, \beta_2})$ ($i = 1,2$) that we have
$$
\operatorname{supp}(\chi_{\epsilon, \beta_1}) \cap
\operatorname{supp}(\chi_{\epsilon, \beta_2}) \subset
W_{\epsilon, \beta_1, l_{\epsilon,\beta_1}} \cap 
W_{\epsilon, \beta_2, l_{\epsilon,\beta_2}} = \varnothing.
$$
Therefore we may assume 
$\operatorname{dim}{A_{\beta_1}} < \operatorname{dim}{A_{\beta_2}}$.
Since
$\beta_1 \npreceq \beta_2$ implies $A_{\beta_1} \cap \bar{A}_{\beta_2} = \varnothing$,
the relations
$$ 
\operatorname{supp}(\chi_{\epsilon, \beta_1}) \cap
\operatorname{supp}(\chi_{\epsilon, \beta_2}) \subset
W_{\epsilon, \beta_1, l_{\epsilon,\beta_1}} \cap
W_{\epsilon, \beta_2, l_{\epsilon,\beta_2}} = \varnothing
$$
follow from the condition 3.  

The, from \eqref{eq:prod=0} and \eqref{eq:prod2=0}, we obtain that
$$
1=\sum_{i=1}^{\#\Lambda}
\sum_{\beta_1 \prec \beta_2 \prec \dots \prec \beta_i\in\Lambda}
(-1)^i\chi_{\epsilon, \beta_1}(x)
\chi_{\epsilon, \beta_2}(x) \dots
\chi_{\epsilon, \beta_i}(x) 
$$
in a neighborhood of $A$. Here $\beta_1 \prec \beta_2$ implies that
$\beta_1 \ne \beta_2$ and $A_{\beta_1} \subset \bar{A}_{\beta_2}$.

\

\noindent
Let $F \in \Sswj(X)$.  Since we have $\Sswj(X) = \Ssawj(X)$, then, for
any $\alpha$ there exists $g_\alpha(x) \in {\ms C}^{(s)}(X)$ such that
$j_{\bar{A}_\alpha} (g_\alpha) = F\vert_{\bar{A}_\alpha}$.  Then, by
Proposition \ref{prop:decay-estimate}, there exists a super growth
indicator $\psi(t)$ such that, for any $h > 0$ and any $\alpha$, there
exists a constant $C'_h$ such that, for any $l > 0$,
$$
\vert\vert j_A(g_\alpha) - F \vert\vert_{B(A_\alpha,\, l), s, h}
\le C'_h \exp(-\psi(l^{-\sigma})) \ .
$$

Set
$$
g_\epsilon(x) := 
\sum_{i=1}^{\#\Lambda}
\sum_{\beta_1 \prec \beta_2 \prec \dots \prec \beta_i\in\Lambda}
(-1)^i\chi_{\epsilon, \beta_1}(x)
\chi_{\epsilon, \beta_2}(x) \dots
\chi_{\epsilon, \beta_i}(x) g_{\beta_i}(x) \in {\ms C}^{(s)}(X) \ .
$$
We are going to show that $j_A(g_\epsilon)$ converges to $F$ with
respect to the topology of $\Sswj(X)$ when $\epsilon \to 0$. This will
complete the proof.

Let us f\mbox{}ix a stratum $A_\alpha$.  We have
$$
\begin{aligned}
&j_{\bar{A}_\alpha}(g_\epsilon) - F\vert_{\bar{A}_\alpha}= \\
&\qquad =j_{\bar{A}_\alpha}(g_\epsilon) - 
\sum_{i=1}^{\#\Lambda}
\sum_{\beta_1 \prec \beta_2 \prec \dots \prec \beta_i\in\Lambda}
j_{\bar{A}_\alpha}
\left((-1)^i\chi_{\epsilon, \beta_1}(x)
\chi_{\epsilon, \beta_2}(x) \dots
\chi_{\epsilon, \beta_i}(x) \right) 
F\vert_{\bar{A}_\alpha} \\
&\qquad =
\sum_{i=1}^{\#\Lambda}
\sum_{\beta_1 \prec \beta_2 \prec \dots \prec \beta_i\in\Lambda}
j_{\bar{A}_\alpha}
\left((-1)^i\chi_{\epsilon, \beta_1}(x)
\chi_{\epsilon, \beta_2}(x) \dots
\chi_{\epsilon, \beta_i}(x) \right) 
(j_{A}(g_{\beta_i}) - F)\vert_{\bar{A}_\alpha}.
\end{aligned}
$$
Noticing 
$$
\operatorname{supp}(\chi_{\epsilon, \beta_i}) \cap A \subset
W_{\epsilon, \beta_i, l_{\epsilon, \beta_i}} \cap A \subset 
U(A_{\beta_i})
$$
and continuity, we get
$$
j_{\bar{A}_\alpha}
\left((-1)^i\chi_{\epsilon, \beta_1}(x)
\chi_{\epsilon, \beta_2}(x) \dots
\chi_{\epsilon, \beta_i}(x) \right) \equiv 0
$$
if $U(A_{\beta_i}) \cap A_\alpha = \varnothing$. Since $U(A_{\beta_i})
\cap A_\alpha \ne \varnothing$ implies $\bar{A}_\alpha \supset
A_{\beta_i}$, we obtain
$$
\begin{aligned}
&j_{\bar{A}_\alpha}(g_\epsilon) - F\vert_{\bar{A}_\alpha}= \\
&\qquad =
\sum_{i=1}^{\#\Lambda}
\sum_{\beta_1 \prec \beta_2 \prec \dots \prec \beta_i \preceq A_\alpha}
j_{\bar{A}_\alpha}
\left((-1)^i\chi_{\epsilon, \beta_1}(x)
\chi_{\epsilon, \beta_2}(x) \dots
\chi_{\epsilon, \beta_i}(x) \right) 
(j_{A}(g_{\beta_i}) - F)\vert_{\bar{A}_\alpha}.
\end{aligned}
$$
Now, let $\beta_1$, $\dots$, $\beta_i$ be such that 
$\beta_1 \prec \beta_2 \prec \dots \prec \beta_i \preceq A_\alpha$.
As $l_{\epsilon,\beta_1} > l_{\epsilon,\beta_2} > \dots > l_{\epsilon,\beta_i}$,
we have that, for any $h$,
$$
\vert\vert
j_A\left((-1)^i\chi_{\epsilon, \beta_1}(x)
\chi_{\epsilon, \beta_2}(x) \dots
\chi_{\epsilon, \beta_i}(x) \right) 
\vert\vert_{A,s,\#\Lambda h}
\le 
C_h^{\#\Lambda}
\exp\left(\#\Lambda\varphi(l_{\epsilon,\beta_i}^{-\sigma})\right).
$$
We also have, for any $h>0$,
$$
\vert\vert j_A(g_{\beta_i}) - F 
\vert\vert_{B(A_{\beta_i},\, l_{\epsilon, \beta_i}), s, h}
\le C'_h \exp(-\psi(l_{\epsilon, \beta_i}^{-\sigma})).
$$
Since we have 
$W_{\epsilon, \beta_i, l_{\epsilon, \beta_i}} \cap A
\subset 
W_{\epsilon, \beta_i, l_{\epsilon, \beta_i}} \cap 
U(A_{\beta_i}) \subset B(A_{\beta_i},\,l_{\epsilon, \beta_i})$,
we get
$$
\operatorname{supp}(\chi_{\epsilon, {\beta_i}}) \cap \bar{A}_\alpha
\subset W_{\epsilon, \beta_i, l_{\epsilon, \beta_i}} \cap \bar{A}_\alpha
\subset B(A_{\beta_i},\, l_{\epsilon, \beta_i}).
$$
Hence, for any $h>0$, we obtained
$$
\begin{aligned}
&\vert\vert j_{\bar{A}_\alpha}
\left((-1)^i\chi_{\epsilon, \beta_1}(x)
\chi_{\epsilon, \beta_2}(x) \dots
\chi_{\epsilon, \beta_i}(x) \right) 
(j_{A}(g_{\beta_i}) - F)\vert_{\bar{A}_\alpha}
\vert\vert_{\bar{A}_\alpha, s, (\#\Lambda+1)h} \\
&
\qquad\le 
\vert\vert j_{\bar{A}_\alpha}
\left((-1)^i\chi_{\epsilon, \beta_1}(x)
\chi_{\epsilon, \beta_2}(x) \dots
\chi_{\epsilon, \beta_i}(x) \right) 
\vert\vert_{B(A_{\beta_i},\,l_{\epsilon, \beta_i}) \cap \bar{A}_\alpha, s, \#\Lambda h} \\
&\qquad\qquad\qquad\qquad \cdot \vert\vert 
(j_{A}(g_{\beta_i}) - F)\vert_{\bar{A}_\alpha}
\vert\vert_{B(A_{\beta_i},\, l_{\epsilon, \beta_i}) \cap \bar{A}_\alpha, s, h} \\
&\qquad \le
C_h^{\#\Lambda}C'_h \exp(\#\Lambda\varphi(l_{\epsilon, \beta_i}^{-\sigma}) 
-\psi(l_{\epsilon, \beta_i}^{-\sigma})).
\end{aligned}
$$
In the end, if we take a super growth indicator $\varphi(t)$ such that 
$$
\varphi(t) \ll \frac{1}{2\#\Lambda}\psi(t) \ .
$$ 

Then, we obtain
$$
\begin{aligned}
&\vert\vert 
j_{\bar{A}_\alpha}
\left((-1)^i\chi_{\epsilon, \beta_1}(x)
\chi_{\epsilon, \beta_2}(x) \dots
\chi_{\epsilon, \beta_i}(x) \right) 
(j_{A}(g_{\beta_i}) - F)\vert_{\bar{A}_\alpha}
\vert\vert_{\bar{A}_{\alpha}, s, (\#\Lambda+1)h} \\
&\qquad 
\le
CC_h^{\#\Lambda}C'_h \exp\left(-\frac{1}{2}\psi(l_{\epsilon, \beta_i}^{-\sigma})\right) 
\to 0\qquad (\epsilon \to 0)
\end{aligned}
$$
for any $h>0$.
\end{proof}

\addcontentsline{toc}{section}{\textbf{References}}
\bibliography{biblio}{}
\bibliographystyle{abbrv}

\vspace{10mm}

{\small{\sc Naofumi Honda

Department of Mathematics, Faculty of Science,

Hokkaido University

Kita 10, Nishi 8, Kita-Ku, Sapporo, Hokkaido, 060-0810, Japan

}

E-mail address: $\texttt{honda@math.sci.hokudai.ac.jp}$}

\vspace{5mm}

{\small{\sc Giovanni Morando

Dipartimento di Matematica Pura ed Applicata,

Universit{\`a} di Padova,

Via Trieste 63, 35121 Padova, Italy.

}and{\sc

Centro de {\'A}lgebra da Universidade de Lisboa,

Av. Prof. Gama Pinto, 2, 1649-003 Lisboa, Portugal.

}

E-mail address: $\texttt{gmorando@math.unipd.it}$}

\end{document}